\documentclass[a4paper,11pt]{article}
\usepackage[utf8]{inputenc}

\title{The joint fluctuations of the lengths of the Beta$(2-\alpha, \alpha)$-coalescents
}
\author{Matthias Birkner\thanks{Institut f\"ur Mathematik, Johannes Gutenberg-Universit\"at Mainz, 55128 Mainz, Germany, \hspace{2cm} \mbox{} \mbox{birkner@uni-mainz.de}, idahmer@uni-mainz.de} , Iulia Dahmer$^*$, Christina S. Diehl\thanks{Institut f\"ur Mathematik, Goethe-Universit\"at Frankfurt, 60325 Frankfurt, Germany, \hspace{3cm} \mbox{} \mbox{ch.s.diehl@gmail.com}, kersting@math.uni-frankfurt.de} ,   G\"{o}tz Kersting$^\dag$}
\date{\today}

\usepackage{amsfonts}
\usepackage{amsmath}
\usepackage{amssymb}
\usepackage{bbm,ulem}
\usepackage{amsthm}
\usepackage{fullpage}
\usepackage{url}
\usepackage{mathrsfs}
\usepackage{graphicx,rotating}
\usepackage{xcolor}
\usepackage{dsfont}
\usepackage{hyperref}

\newcommand{\E}{\ensuremath{\mathbb{E}}}
\newcommand{\Prob}[1]{\ensuremath{\mathbb{P} \left(#1 \right)}}

\newcommand{\R}{\ensuremath{\mathbb{R}}}

\newcommand{\N}{\ensuremath{\mathbb{N}}}
\renewcommand{\P}{\ensuremath{\mathbb{P}}}

\newcommand{\1}{\mathds{1}}

\newtheorem{theorem}{Theorem}[section]
\newtheorem{lemma}[theorem]{Lemma}
\newtheorem{corollary}[theorem]{Corollary}

\newtheorem{remark}[theorem]{Remark}

\newcommand{\al}{\ensuremath{\alpha}}
\newcommand{\g}{\ensuremath{\gamma}}
\newcommand{\D}{\ensuremath{\Delta}}

\begin{document}
\maketitle

\abstract{ 
We consider Beta$(2-\alpha, \alpha)$-coalescents with parameter range $1 <\alpha<2$ starting from $n$ leaves.  The length $\ell^{(n)}_r$ of order $r$ in the $n$-Beta$(2-\alpha, \alpha)$-coalescent tree is defined as the sum of the lengths of all branches that carry a subtree with $r$ leaves. We show that for any $s \in \mathbb N$ the vector of suitably centered and rescaled lengths of orders $1\le r \le s$ converges in distribution to a multivariate stable distribution as the number of leaves tends to infinity. 
}

\section{Introduction and main result}

Multiple merger coalescents, also known as $\Lambda$-coalescents, are
partition-valued Markov processes in continuous time where in each jump a random number
of classes merges into a single class. They are a natural
generalisation of the famous Kingman coalescent \cite{K82}, where only
pairs of classes merge. $\Lambda$-coalescents were introduced in 1999 by Donnelly and
Kurtz \cite{DK99}, by Pitman \cite{P99} and by Sagitov \cite{S99}, motivated
both by applications to genealogies in stochastic population models and by
their rich mathematical structure. They have since been an object of intense
study, see for example the surveys \cite{B09} and
\cite{GIM14}; we refer also to the recent overview articles \cite{BB19} and \cite{KW20},
which have a focus on their role in mathematical population genetics.

In this article, we consider the subclass where $\Lambda=\mathrm{Beta}(2-\alpha, \alpha)$
is a Beta-measure and $1 < \alpha < 2$.
This class appears naturally as limiting genealogies in certain population models
with infinite offspring variance, see \cite{S03}, and is also closely connected to
continuous state branching processes via a time-change, see \cite{BBCEMSW05}.
It is in a sense prototypical for $\Lambda$-coalescents where $\Lambda$ has a
density with singularity of the form $\sim c x^{1-\alpha}$ at $0+$.

For a sample of size $n$, an $n$-$\mathrm{Beta}(2-\alpha, \alpha)$-coalescent can be visualised as a random tree with $n$
leaves. Functionals of this tree like the total length $L_n$ and also
the internal length $\ell_r^{(n)}$ of order $r$, the length of all
branches subtending exactly $r$ leaves for $r=1,2,\dots,n-1$, are
mathematically interesting and also important in population genetics
applications because of their close relation to the so-called site
frequency spectrum: A mutation which appeared in the part of the
genealogical tree which contributes to $\ell_r^{(n)}$ will be present
in $r$ out of the $n$ samples.  Asymptotic properties of these
functionals have been studied concerning their typical growth rates
(see e.g.\ \cite{BBS07}, \cite{BBL14}), and also some results
concerning their fluctuations have been obtained, see \cite{K12},
\cite{DKW14}, \cite{LT15}, \cite{SJY16}, \cite{DY15}. For more general classes $\Lambda$-coalescents see also results in \cite{DK19a}, \cite{DK19b}, \cite{DIMR07}, \cite{GIMM14}, \cite{M06}. We continue these studies for the case of the Beta-coalescents, extending the
approach in \cite{DKW14} to the analysis of the joint fluctuations of the
random vector $(\ell_1^{(n)}, \ell_2^{(n)}, \dots, \ell_r^{(n)})$ from
$r=1$ to any fixed $r \in \N$, see our main result  Theorem \ref{theorem} below.

We exclude in this study the two boundary cases $\alpha=1$, the so-called Bolthausen-Sznitman coalescent, and
$\alpha=2$, which corresponds to Kingman's coalescent (a result corresponding to our
main result was already obtained in \cite{DK16} for this case). 
\smallskip

Let $(\Pi^{(n)}_t)_{t \ge 0}$ be an $n$-Beta$(2-\alpha, \alpha)$-coalescent with $1<\al<2$, denote by $H_n=\inf\{ t \ge 0 : \# \Pi^{(n)}_t = 1 \}$ the height of the coalescent tree, and for $r \in \N$ 
\begin{equation}
    \ell_r = \ell^{(n)}_r 
    = \int_0^{H_n} \big(\text{number of blocks in $\Pi^{(n)}_t$ with $r$ elements}\big) \ dt.
\end{equation}

Our main result describes the joint fluctuations of these lengths:

\begin{theorem}\label{theorem}
For each $r \in \mathbb N$, as $n \to \infty$
\begin{align*}
\Big( \frac{ c_1 n^{2-\alpha}-\ell^{(n)}_1} {n^{1-\alpha +\frac 1 \alpha}} , \dots&, \frac{c_r n^{2-\alpha}- \ell^{(n)}_r } {n^{1-\alpha +\frac 1 \alpha}}\Big)  
\\
&\stackrel{d}{\longrightarrow}  \Big(\mathcal S_{\frac 1 \gamma}, \int_0^{\frac 1 \gamma } (1-\gamma t)^{\alpha-1} d\mathcal S_t, \dots, \int_0^{\frac 1 \gamma } (1-\gamma t)^{(\alpha-1)(r-1)}{d\mathcal S_t}\Big)\times  R 
\end{align*}
where $\gamma=\frac 1 {\alpha-1}$, $c_1, \dots, c_r$ are constants, $R$ is an $(r\times r)$
upper triangular matrix of constant coefficients and
$\mathcal S=(\mathcal S_t)_{t \geq 0}$ is a stable L\'evy process with index
$\al$ normalized by the properties
\begin{equation}\label{eq:prop_stable}
\mathbb E[\mathcal S_1]=0, \qquad \mathbb P(\mathcal S_1>x) \sim   b_\alpha x^{-\alpha}, \qquad \mathbb P(\mathcal S_1<-x) = o(x^{-\alpha})
\end{equation}
for $x \to \infty$ with $b_\alpha= 1/\Gamma(2-\alpha)$
(See Remark \ref{rem:equiv_char} for equivalent characterisations.) 
\end{theorem}
In particular, the random vector $\Big(\mathcal S_{\frac 1 \gamma}, \int_0^{\frac 1 \gamma } (1-\gamma t)^{\alpha-1} d\mathcal S_t, \dots, \int_0^{\frac 1 \gamma } (1-\gamma t)^{(\alpha-1)(r-1)}{d\mathcal S_t}\Big)\times  R $ has a multivariate stable distribution.

The coefficients  \[c_r= \alpha(\alpha-1)^2 \frac{\Gamma(r+\alpha-2)}{r!}\]
were already obtained in \cite[Thm.~9]{BBS07}. In particular we recover their Theorem 9 with a different proof, see also Remark \ref{rm:coefficients} below.

For any $1\leq m< s\leq r$ consider the compositions $(r_1, \dots, r_m) \in \mathbb N^m$ of length $m$ of the number $s-1$, i.e. $\sum_{i=1}^m r_i=s-1$. Denote $
   \widehat{r}_j = s - \sum_{i=1}^j r_i = r_{j+1}+r_{j+2}+\cdots+r_m+1
$ 
(read $\widehat{r}_0 = s, \,\widehat{r}_m=1$).
 In particular for $1\leq p\leq m$, $r_p$ and $\widehat r_p$ depend on $s$ and $m$. Denote also for
$k \in \{1,2,\dots,m\}$
\begin{align}\label{C_k}
& C_k \big((r_1,\dots,r_m)\big)=\frac {\alpha-1 }{s+\frac {2-\alpha} {\alpha-1}} \cdot \prod_{j=1}^{k-1} \frac 1 {\widehat r_{j-1} -\widehat r_{k-1}} 
\\
&\hspace{1cm}\cdot \Big(\frac {(-1)^{m-k+1}}{\prod_{j=k}^{m} (\widehat  r_{k-1}- \widehat r_j)}
 + \frac 1 {s-1}\sum_{q=k}^m (-1)^{q-k} \cdot \frac{s- r_{q}+ \frac {2-\alpha} {\alpha-1}} {\prod_{j=q+1}^m (\widehat r_{j-1}+ \frac {2-\alpha} {\alpha-1} )} \notag
 \cdot  \frac{\widehat r_{k-1} -\widehat r_q} {\prod_{j=k}^q (\widehat r_{j-1}-\widehat r_k)} \Big).\notag
\end{align} 

With this notation, the entries of the matrix $R$ are given as 
\begin{equation}
  \label{R11}
    R_{1,1} = \alpha(\alpha-1)(2-\alpha) \Gamma(\alpha)
\end{equation}
and for $2\leq s\leq r$ and $1<j \le r$ by
\begin{align}
& R_{1, s}= \al^2 (\alpha-1) \Gamma(\al)  \sum_{(r_1, \dots, r_m)}\prod_{p=1}^m \frac{\widehat r_p \cdot \frac {\al}{\Gamma(2-\al)} \frac{\Gamma (r_p+1-\al)}{\Gamma(r_p+2)} }{\widehat r_{p-1} -1},\label{R_1} 
\\
& R_{j, s}= \al \Gamma(\al)  \sum_{(r_1, \dots, r_m)} \bigg(\prod_{p=1}^m \Big(\widehat r_p \frac {\al}{\Gamma(2-\al)} \frac{\Gamma (r_p+1-\al)}{\Gamma(r_p+2)} \Big) \notag
\\
& \hspace{4.5cm}\cdot\sum_{k=1}^m  \Big(\delta_{\widehat r_{k-1},j} C_k \big((r_1,\dots,r_m)\big) \Big)\bigg) \cdot \1_{1\leq j\leq s},
\label{R_j}\end{align}
where the multiple sums run over $m=1, \dots, s-1$ and all compositions $(r_1, \dots, r_m)\in \mathbb N^m$ of $s-1$ as above, $\delta_{i,k}$ denoting the Kronecker symbol.

 Note that by using the functional equation of the gamma function one can see that all the coefficients $R_{j,s}$ are of the form $\Gamma(\alpha)$ times a rational function of $\alpha$.

\begin{remark}\label{rem:equiv_char}
The characteristic function of $\mathcal S_1$ is given by
\[\mathbb E [\exp(i \theta \mathcal S_1)]= \exp\Big(-\sigma^\alpha |\theta|^\alpha \big(1-i \big(\mathrm{sign } \,\theta \big) \tan \frac{\pi\alpha} 2\big) \Big), \quad \theta \in \mathbb R,\]
with $\sigma= \big(-\cos(\frac{\pi \alpha }2)/(\alpha-1)\big)^{1/\alpha}$. In the notation in \cite{ST94} (see in particular Property 1.2.15 therein) the distribution of $\mathcal S_1$ can be equivalently characterised as
\[\mathcal S_1 \sim S_\alpha (\sigma, 1,0),\] 
where $S_\alpha (\sigma, \beta, \mu)$ denotes the stable law with index $\alpha$, scale parameter $\sigma$, skewness parameter $\beta$ and shift parameter $\mu$.
\end{remark}
	\smallskip
	Let 
	\begin{equation}
   L^{(n)}
    = \int_0^{H_n} \big(\text{number of blocks in $\Pi^{(n)}_t$ 
		}\big) \ dt
\end{equation}
	be the total length of the $n$-Beta$(2-\alpha, \alpha)$-coalescent. The following result combines Theorem~\ref{theorem} and the main result from \cite{K12}.
	
	\begin{corollary}\label{corollary}
 For $1<\alpha<\frac{1+\sqrt 5} 2$ it holds that for $r \in \mathbb N$
	as $n \to \infty$
\begin{align*}
\Big( \frac{ c_1 n^{2-\alpha}-\ell^{(n)}_1} {n^{1-\alpha +\frac 1 \alpha}}&, \dots, \frac{c_r n^{2-\alpha}- \ell^{(n)}_r } {n^{1-\alpha +\frac 1 \alpha}}, \frac{ \widetilde c n^{2-\alpha}-L^{(n)}} {n^{1-\alpha +\frac 1 \alpha}} \Big)  
\\
&\stackrel{d}{\longrightarrow}  \Big(\mathcal S_{\frac 1 \gamma}, \int_0^{\frac 1 \gamma } (1-\gamma t)^{\alpha-1} d\mathcal S_t, \dots, \int_0^{\frac 1 \gamma } (1-\gamma t)^{(\alpha-1)(r-1)}{d\mathcal S_t}, \int_{0}^{\frac 1 \gamma} t^{1-\alpha} d \mathcal S_t \Big)\times  \widetilde R 
\end{align*}
	with $\widetilde c=\frac{\Gamma(\alpha) \alpha(\alpha-1)}{(2-\alpha)}$ and $\widetilde R= \begin{pmatrix}
  &  & & 0\\
	& R & & \vdots\\
	&   &  & 0 \\
0 & \hdots & 0 &	\Gamma(\alpha)\alpha(\alpha-1)^{1+\frac 1 \alpha} 
\end{pmatrix}.$
	\end{corollary}
Observe by \cite[Theorem~2.1]{KM88} and formula \eqref{last} below that the integral $\int_{0}^{\frac 1 \gamma} t^{1-\alpha} d \mathcal S_t$  is well-defined when $1<\alpha<\frac {1+\sqrt 5} 2$.

\smallskip
Furthermore, for $1\leq k\leq r$ with $f_k(t)=(1-\gamma t)^{(k-1)(\alpha-1)}\cdot \1_{(0,1/\gamma)}(t)$, and for  $f_{r+1}(t)=t^{1-\alpha}\cdot \1_{(0,1/\gamma)}(t)$ we have
\[I(f_k):=\int_\mathbb R f_k(t) d\mathcal S_t\sim S_{\alpha} \Big(\sigma\Big(\int_\mathbb R (f_k(t))^\alpha  dt\Big)^{1/\alpha}, 1,0\Big),\]
see \cite{ST94} Proposition 3.4.1 (write $I(f_k)=\int_\mathbb R f_k(t) M(dx) $ with random measure $M$ having control measure $m=\sigma^\alpha \cdot \lambda$, $\lambda$ the Lebesgue measure, and skewness intensity $\beta\equiv 1$). By \cite{ST94} Proposition 3.4.2 the limiting multivariate stable distribution from Corollary \ref{corollary} has the  characteristic function
\begin{align*}
&\phi_{f_1,\dots, f_{r+1}}(\theta_1, \dots, \theta_{r+1})=
\\
& \exp \Big\{ -\int_0^{1/\gamma} \Big|\sum_{j=1}^{r+1} \theta_j \big(\sum_{i=1}^j \widetilde R_{ij} f_j(t)\big)\Big|^\alpha 
 \big(1-i\,\mathrm{sign}\big(\sum_{j=1}^{r+1} \theta_j \big(\sum_{i=1}^j \widetilde R_{ij} f_j(t)\big)\big)\big) \cdot \tan\Big(\frac{\alpha\pi} 2\Big)\cdot \sigma^\alpha dt\Big\}.
\end{align*}
 
 Note that by putting $\theta_{r+1}=0$ we cover the limiting distribution from Theorem \ref{theorem} without further restriction on $\alpha$.

\begin{remark}
In view of \cite[Thm.~1,~(iii)]{K12}, we expect the  the following analogue of Corollary~\ref{corollary} for $\frac{1+\sqrt 5} 2<\alpha<2$: For $r \in \mathbb N$
	as $n \to \infty$
\begin{align*}
\Big( \frac{ c_1 n^{2-\alpha}-\ell^{(n)}_1} {n^{1-\alpha +\frac 1 \alpha}}&, \dots, \frac{c_r n^{2-\alpha}- \ell^{(n)}_r } {n^{1-\alpha +\frac 1 \alpha}}, \widetilde c n^{2-\alpha}-L^{(n)}\Big)  
\\
&\stackrel{d}{\longrightarrow}  \Big(\mathcal S_{\frac 1 \gamma}, \int_0^{\frac 1 \gamma } (1-\gamma t)^{\alpha-1} d\mathcal S_t, \dots, \int_0^{\frac 1 \gamma } (1-\gamma t)^{(\alpha-1)(r-1)}{d\mathcal S_t}, \eta \Big)\times  \bar R 
\end{align*}
	with $\eta$ a nondegenerate random variable independent of $\mathcal S$ and $\bar R= \begin{pmatrix}
  &  & & 0\\
	& R & & \vdots\\
	&   &  & 0 \\
0 & \hdots & 0 &	1 
\end{pmatrix}.$
\end{remark}

	\begin{remark}[Special cases]
We give here the exact expressions for the special cases $r=2$ and $r=3$. 
\begin{enumerate}
\item
For the length of order 2 it holds as $n \to \infty$ that
\begin{align*}
 \frac{c_2 n^{2-\alpha}- \ell^{(n)}_2 } {n^{1-\alpha +\frac 1 \alpha}} &\stackrel{d}{\longrightarrow}  R_{1 2}\mathcal S_{\frac 1 \gamma}+ R_{22}\int_0^{\frac 1 \gamma } (1-\gamma t)^{\alpha-1} d\mathcal S_t, 
\end{align*}
with 
\[R_{12}=\frac {\alpha^3(\alpha-1)\Gamma(\alpha)}{2} \quad \text{and} \quad R_{22}= \frac {\alpha(\alpha-1)(2-\alpha)\Gamma(\alpha)}{2}.\]

\item
For the length of order 3 it holds as $n \to \infty$ that
\begin{align*}
 \frac{c_3 n^{2-\alpha}- \ell^{(n)}_3 } {n^{1-\alpha +\frac 1 \alpha}} &\stackrel{d}{\longrightarrow}  R_{1 3}\mathcal S_{\frac 1 \gamma}+ R_{23}\int_0^{\frac 1 \gamma } (1-\gamma t)^{\alpha-1} d\mathcal S_t + R_{33}\int_0^{\frac 1 \gamma } (1-\gamma t)^{2(\alpha-1)} d\mathcal S_t, 
\end{align*}
where 
\begin{align*}
& R_{13}=  \frac {\alpha^3(\alpha-1)(\alpha+1) \Gamma(\alpha)} {6},
\\
& R_{23}= \frac {\alpha^3(\alpha-1)(2-\alpha) \Gamma(\alpha)} {4(2\alpha-1)} ,
\\
& R_{33}= \frac {\alpha^2 (\alpha-1)\Gamma(\alpha)} {12(2\alpha-1)^2}\cdot (5\alpha^3-18\alpha^2+15\alpha-4).
\end{align*}
\end{enumerate}

\end{remark}

\medskip

For a sample of size $n$ out of a population evolving under the infinitely many alleles model and whose genealogy is given by a $\mathrm{Beta}(2-\alpha, \alpha)$-coalescent the site frequency spectrum (SFS) is a widely used statistic which summarises the genetic data. 
Given the coalescent tree, the entries $\xi^{(n)}_1, \dots, \xi^{(n)}_{n-1}$ of the SFS, i.e. the numbers of mutations carried by exactly $1,2,\dots, n-1$ individuals in the sample, are Poisson distributed with parameters $\theta \ell^{(n)}_1, \dots, \theta \ell^{(n)}_{n-1}$, where $\theta$ is the mutation rate. Note from Theorem \ref{theorem} that the Poisson fluctuations of $\xi^{(n)}_r$, $1\leq r\leq n-1$, which are of order $n^{\frac {2-\alpha} 2}$ are dominated by the fluctuations of the internal length (which are of order $n^{1-\alpha+1/\alpha}$) as long as $\alpha <\sqrt 2$. Koskela \cite{K18} proposed a statistical test to distinguish between the Kingman coalescent and $\Lambda$-coalescents based on the two-dimensional summary statistic 
\[\Big(\frac{\xi^{(n)}_{1}}{\sum_{i=1}^{n-1} \xi^{(n)}_{i}}, \, \frac{\sum_{j=k}^{n-1}\xi^{(n)}_{j}}{\sum_{i=1}^{n-1}\xi^{(n)}_{i}}\Big)\]
for some fixed specified $3\leq k\leq n-1$. Our Corollary \ref{corollary} allows one to deduce the asymptotic fluctuations of this statistic which can be read off those of $(\ell^{(n)}_{1},  L^{(n)} - \sum_{j=1}^{k-1} \ell^{(n)}_{j}, L^{(n)})$.

\medskip

The proof of our main result relies on the representation of the length $\ell^{(n)}_{r}$ as the sum of the numbers $Z_{r,k}$ of blocks in the coalescent after $k$ mergers, which contain $r$ elements, multiplied by the inter-coalescence times between the $k$-th and the $(k+1)$-th merger. We show that the exponential inter-coalescence times contribute to the fluctuations of the lengths only through their rates which are essentially given by the number $X_k$ of blocks in the coalescent after $k$ mergers (i.e.\ the state of the block counting process at time $k$) to the power $\alpha$. On the other hand, the numbers $Z_{r,k}$, which encode the sampling randomness from the population, bring in a contribution to the fluctuations of the lengths only through their conditional expectations given the block counting process $X$. Thus, the length can be expressed as a large combinatorial expression involving only the block counting process, see $\bar\ell_r$ in Lemma~\ref{lm:indicators_new}, which is asymptotically re-expressed in terms of multiple integrals in Lemmas~\ref{lm:L2} and \ref{lm:L1}. This expression together with the key observation that the fluctuations of $X_k$ around its typical value are asymptotically stable lead to our result. Section~\ref{sect:Approx} of the paper contains 
preparations which allow
to approximate $\ell^{(n)}_r$ step by step separating the asymptotically negligible contributions, as discussed. These results are combined in Section~\ref{sect:Proof} to complete the proof of Theorem~\ref{theorem} and of Corollary~\ref{corollary}.

In all the proofs we interpret empty sums as 0 and empty products as 1.

\section{Approximations}	
\label{sect:Approx}

\subsection{Preparations}\label{Preparations}

Denote by $\tau_n$ the number of jumps, 
$0 = T_0 < T_1 < \cdots < T_{\tau_n}$ the jump times, and
\begin{equation}
  X_k=X^{(n)}_k = \# \Pi_{T_k}, 
  \quad k=0,1,\dots,\tau_n
\end{equation} 
the number of blocks after $k$ jumps ($X^{(n)}_0=n, X^{(n)}_{\tau_n}=1$, put $X^{(n)}_k=1$ for $k>\tau_n$). These blocks form the (discrete) block-counting process $X=\{X_k\}_{0 \le k \le \tau_n}$. For a $\mathrm{Beta}(2-\alpha, \alpha)$--coalescent this is a decreasing Markov chain with transition probabilities \\$p_{mj}=\binom m{m-j+1} \lambda_{m,m-j+1}/ \lambda_m$, where
\[ \lambda_{m,k}=\frac 1{\Gamma(2-\alpha)\Gamma(\alpha)}\int_0^1 t^{k-\alpha-1}(1-t)^{m-k+\alpha-1}\, dt, \quad k=2, \ldots,m \]
gives the rate at which any fixed $k$-tuple of blocks among $m$ blocks in the coalescent mergers, 
and $\lambda_m=\sum_{k=2}^{m} \binom m k \lambda_{m,k}$.
Let
\begin{equation}
  \label{eq:Delta_k}
\Delta_k=\Delta^{(n)}_k := X^{(n)}_{k-1}-X^{(n)}_{k}, \quad k=1,2,\dots,\tau_n
\end{equation}
be the size of the $k$-th jump of the block counting
process (note that $X$ jumps only downwards, we use the notation
convention \eqref{eq:Delta_k} so that $\Delta_k \ge 1$ for
$k=1,2,\dots,\tau_k$; our definitions imply $\Delta_k = 0$ for
$k > \tau_n$).  Denote $\gamma=1/(\alpha-1)$. Note
\begin{equation}
  \label{eq:jump law asymp}
  \lim_{n\to\infty} \P(\Delta^{(n)}_1 = j )
  = \frac{\al}{\Gamma(2-\al)} \frac{\Gamma(j+1-\al)}{\Gamma(j+2)},
  \quad j =1,2,\dots
\end{equation}
and
\begin{equation}
  \lim_{n\to\infty} \E\big[ \Delta^{(n)}_1\big] = \frac 1 {\alpha-1}=\gamma,
\end{equation}
e.g.\ \cite{DDSJ08}, \cite[Eq.~(4) in Sect.~2]{K12}.
\bigskip

Define
  \begin{equation}\label{eq:S_n}
    \mathcal{S}^{(n)}_t = \frac{n - \gamma \big( \lfloor n t \rfloor \wedge \tau_n \big) - X^{(n)}_{\lfloor n t \rfloor \wedge \tau_n}}{n^{1/\alpha}}
    = \frac{1}{n^{1/\alpha}} \sum_{k=1}^{\lfloor n t \rfloor \wedge \tau_n} (\Delta_k - \gamma), \quad t \ge 0.
  \end{equation}

  \begin{remark}\label{rm:remark}
  \eqref{eq:S_n} yields: 
  \begin{equation}\label{eq:X_k with S_n}
    X_k = n - \gamma k - n^{1/\alpha} \mathcal{S}^{(n)}_{k/n}, \quad k=0,1,\dots, \tau_n.
  \end{equation}
  Applying this with $k=\tau_n$ gives also
  \begin{equation}\label{eq:tau_n via S_n}
    \tau_n = \frac{n-1}{\gamma} - \frac{n^{1/\alpha}}{\gamma} \mathcal{S}^{(n)}_{\tau_n/n}
    = \frac{n}{\gamma} - \frac{n^{1/\alpha}}{\gamma} \mathcal{S}^{(n)}_{1/\gamma} + o_{\P}\big(n^{1/\alpha}\big).
  \end{equation}
  (For the second equality note that the proof of Lemma~\ref{lm:S_n}
  below shows that
  $\mathcal{S}^{(n)}_{1/\gamma} - \mathcal{S}^{(n)}_{\tau_n/n} =
  o_{\P}(1)$.)
  
  We can combine \eqref{eq:X_k with S_n} and \eqref{eq:tau_n via S_n} to find
  \begin{align}
    \label{eq:X_k with S_n alt}
    X_k = 1 + \gamma ( \tau_n -k ) + n^{1/\alpha} \big( \mathcal{S}^{(n)}_{\tau_n/n} - \mathcal{S}^{(n)}_{k/n}\big).
  \end{align}
  
  In particular, we have for $n \to \infty$ that
\begin{align}\label{lm:tau}
\frac {\tau_n- n/\gamma}{n^{1/\alpha}} \mathop{\longrightarrow}^d -\frac {\mathcal S_{1/\gamma}}{\gamma}.\end{align}
Note that \eqref{lm:tau} was already obtained in \cite{GY07} and \cite{DDSJ08}.
\end{remark}

Consider now a random variable $V$ 
with values in $\mathbb N$ and distribution 
\begin{equation}\label{eq:distrib_V}
\mathbb P(V=k)=\frac {\al}{\Gamma(2-\al)} \frac{\Gamma (k+1-\al)}{\Gamma(k+2)}, \quad k\geq 1.
\end{equation}
According to \cite{K12} one can couple the jumps $\Delta_j$ of the block counting process with a sequence $V_1, V_2, \dots$ of independent copies of the random variable $V$ such that (see \cite[(6) and Lemma~3]{K12}) for $j<\tau_n$ we have $\Delta_j \leq V_j$ a.s. and there is a $c>0$ such that
\begin{align}\label{couplingK12_1}
\mathbb P(V\ge k)\le ck^{-\alpha}, \qquad \mathbb P(V_1\neq \Delta_1 | X_0=m) \le cm^{-1}, \quad m \in  \mathbb N
\end{align}
and
\begin{align}\label{couplingK12_2}
\mathbb P(V_1\ge k |V_1\neq \Delta_1) \le ck^{1-\alpha}.
\end{align}

We will make use of this coupling several times in our proofs. For $j \ge \tau_n$, we simply choose $V_j$ as
independent copies of $V$ from \eqref{eq:distrib_V}, independent of everything else.

\begin{lemma}\label{lm:S_n}
  We have 
  \begin{align}
    \big( \mathcal{S}^{(n)}_t )_{t \ge 0}
    \mathop{\Longrightarrow}_{n\to\infty} (\mathcal{S}_{t \wedge (1/\gamma)})_{t \ge 0}
  \end{align}
  (convergence in distribution on the space of c\`adl\`ag paths equipped with the Skorohod \mbox{($J_1$-)} topology)
  where $\mathcal{S}$ is a centered stable L\'evy process of index $\alpha$ which is (in the parametrisation
  from \eqref{eq:S_n}) maximally skewed to the right, as in Theorem \ref{theorem}.

  In particular, the families
  $\big\{ \sup_{t\ge 0} |\mathcal{S}^{(n)}_t| \big\}_{n \in \N}$ and
  $\big\{ \max_{0 \le k \le \tau_n} |X_k^{(n)}-(n-\gamma
  k)|/n^{1/\alpha} \big\}_{n \in \N}$ are tight.
  \end{lemma}
	
 \begin{proof}
Let $(V_k)_{k\ge1}$ be the coupling to $(\Delta_k)_{k \ge 1}$ from above and 
$M:= \sup_{k<\tau_n} | \sum_{j=1}^k(V_j-\Delta_j)|$. Since $\Delta_j \le V_j$ a.s. for $j<\tau_n$, we have
\begin{align*}
 \mathbb E[M] &= \mathbb E[\sum_{j< \tau_n} (V_j-\Delta_j) ] = \mathbb E[\sum_{j=0}^n (V_j-\Delta_j)\1_{X_j>1}] 
\\
&=\sum_{j=0}^n \mathbb E[(V_j-\Delta_j) \1_{X_j>1}] =\sum_{j=0}^n \mathbb E_n[\mathbb E[V_j-\Delta_j|X_j];X_j>1]
\end{align*}
Moreover for $m \in \mathbb N$
\begin{align}
\mathbb E[V_1-\Delta_1 | X_0=m]&\le \mathbb E[V; V \ge m/2] + \mathbb E[V \1_{V\le m/2, V_1\neq \Delta_1} | X_0=m] \notag
\\
&\le \sum_{k \ge m/2} \mathbb P(V\ge k) + \frac m2 \mathbb P(V\ge \frac m2) \notag \\
& \quad + \sum_{k\le m/2} \mathbb P(V_1\ge k |V_1\neq \Delta_1) \mathbb P(V_1\neq \Delta_1 | X_0=m)
\end{align}
From  \eqref{couplingK12_1} and \eqref{couplingK12_2} we obtain for $c'>0$
\[ \mathbb E[V_1-\Delta_1 | X_0=m] \le c' m^{1-\alpha} \]
and
\begin{equation*}
  \mathbb E[M] \le c' \mathbb E[\sum_{j=0}^nX_j^{1-\alpha} ; X_j >1] \le c' \sum_{k=1}^n k^{1-\alpha}=O(n^{2-\alpha})=o(n^{1/\alpha})
\end{equation*}
Therefore
\[ \mathcal S_t^{(n)} =\frac 1{n^{1/\alpha}}\sum_{k=1}^{[nt]\wedge \tau_n}(V_k- \gamma) + o_{\mathbb P}(1) , \]
where the error term applies uniformly in $t\ge 0$. 

Next recall that for any $1/\alpha <\delta< 1$ we have $\tau_n-n/\gamma =o_{\mathbb P}(n^\delta)$
(this follows from \cite[Thm.~7]{GY07} or \cite[Prop.~3.1]{DDSJ08}). On the event $\{\tau_n>\gamma^{-1}n - n^\delta\}$  the summands with $k\le  \gamma^{-1}n-n^{\delta}$ cancel within the next supremum. We obtain
\begin{align*} \sup_t\Big| \sum_{k=1}^{[nt]\wedge \tau_n}&(V_k- \gamma) - \sum_{k=1}^{[nt]\wedge \gamma^{-1}n}(V_k- \gamma)\Big| \\&\le 2 \sup_{ 0 < k \le 2n^\delta} \Big| \sum_{j=1}^k (V_{[\gamma^{-1}n-n^\delta]+j}-\gamma)\Big| + o_{\mathbb P}(n^{1/\alpha}) .
\end{align*}
The right-hand supremum is of order $O_{\mathbb P}((n^\delta)^{1/\alpha})= o_{\mathbb P}(n^{1/\alpha})$ yielding
\begin{align}\label{SandVs}
\mathcal S_t^{(n)}  = \frac 1{n^{1/\alpha}}\sum_{k=1}^{[nt]\wedge \gamma^{-1}n}(V_k- \gamma) + o_{\mathbb P}(1) , 
\end{align}
where again the error term  applies uniformly in $t$. Now our claim follows by means of  the functional limit theorem for stable distributions.

 For the normalization of the limit law note that $n^{-1/\alpha} \sum_{k=1}^n (V_k-\gamma) \Rightarrow \mathcal{S}_1$, thus $\mathcal{S}_1$ is stable with index $\alpha$, centered and maximally skewed to the right. Furthermore, \begin{equation}
 \label{S1tail}
     \mathbb{P}(\mathcal{S}_1 > x) \sim (\Gamma(2-\alpha))^{-1} x^{-\alpha}
     \quad \text{for } x \to \infty
 \end{equation} compare \cite[Eq.\ (3)]{K12} and \cite[Proof of Lemma~10, p.~2100, line 8 from below]{K12}.
 \end{proof}

For $r=1,\dots,n$ let 
\[ 
Z_{r,k} = \# \{ B \in \Pi_{T_k} : |B|=r\}, \quad k=0,1,\dots,\tau_n
\]
be the number of blocks of size $r$ after $k$ jumps 
($Z_{r,0} = n \delta_{r,1}$, $Z_{r,\tau_n} = \delta_{r,n}$). In particular this is the number of branches of order $r$ (i.e., branches subtending exactly $r$ leaves) in the coalescent tree which exist after the $k$-th jump.

We will use the following stochastic representation of $\ell_r$ (and write - for convenience - in what follows equality instead of equality in distribution):
\[\ell_r^{(n)} \stackrel{d}= \sum_{k=0}^{\tau_n-1} Z_{r,k} W_k /\lambda_{X_k}, \]
where $W_0, W_1, \dots$ are i.i.d. standard exponential distributed random variables and independent of the block counting process $X$ and $\lambda_m=\binom m 2 \lambda_{m,2}+ \cdots + \binom m m \lambda_{m,m}$ denotes for $m \geq 2$ the rate at which a coalescence happens when there are $m$ lineages in the tree.

From Lemma 2.2 in \cite{DDSJ08} we have that as $m \to \infty$
\begin{equation}
\label{eq:lambda_m}
    \lambda_m=\frac 1 {\al\Gamma(\al)}m^\al + O(m^{\al-1}).
\end{equation}

We will make the following approximations (we leave out the superscript $(n)$ for the lengths from now on): 
\begin{align}\label{eq:approximations}
\ell_r = \sum_{k=0}^{\tau_n-1} \frac{ Z_{r,k} }{ \lambda_{X_k}}  W_k  & \approx \sum_{k=0}^{\tau_n-1} \frac{ Z_{r,k} }{\lambda_{X_k}} \approx  \sum_{k=0}^{\tau_n-1} \frac{ \E[Z_{r,k} | X] }{\lambda_{X_k}}
 \approx 
\al\Gamma(\al) \sum_{k=0}^{\tau_n-1} \frac{ \E[Z_{r,k} | X] }{ X_k^\al} .
 \end{align}
Note that by an argument  analogous  to the one in \cite[p.~1024--1025]{DKW14} one has 
\begin{align*}
\ell_r =  \sum_{k=0}^{\tau_n-1} \frac{ Z_{r,k} }{\lambda_{X_k}} + o_\mathbb P (n^{1-\alpha +1/\alpha}),
\end{align*}
which justifies the first approximation. It remains to show the last two. Essentially this means that in our context fluctuations of the block counting process are dominant.

\medskip
Let $[n]=\{1,2, \dots, n\}$. For $0\leq j< k< \tau_n$ and $r \in [n]$ let 
\begin{align}\label{eq:Pi}\Pi_j^k(r):=\prod_{i=j+1}^k \Big(1-\frac r {X_i}\Big)
\end{align}
and $\Pi_j^k=\Pi_j^k(1)$. We use this notation that was already introduced in \cite{DKW14} for consistency. Note in particular that the upper index is not a power, but just an index. Whenever we deal with a power we will write e.g. $(\Pi_j^k)^r$. The following lemma follows directly from results in \cite{DKW14}.

\begin{lemma}\label{lm:DKW}
Let $\varepsilon>0$. It holds uniformly for $0\leq j< k< \tau_n$:\\ 
$\text{i) }$ For $\beta \in \mathbb R$
\begin{align*}
& X_k^{\beta}=\Big(\gamma (\tau_n-k)\Big)^{\beta} \Big(1+
O_{\mathbb P} \Big((\tau_n-k)^{1/\alpha -1 +\varepsilon}\Big)\Big).
\end{align*}
\text{ii) }
\begin{align*}
&  \Pi_j^k=\Big(\frac {\tau_n-k}{\tau_n-j}\Big)^{\alpha-1} \Big(1+
O_{\mathbb P} \Big((\tau_n-k)^{1/\alpha -1 +\varepsilon}\Big)\Big).
\end{align*}
\text{iii)}
\begin{align*}
\Pi_j^k(r) = \big(\Pi_j^k\big)^r \bigg(1 + O_{\mathbb P}\Big(\frac 1 {X_k}\Big) \bigg) \quad \text{and} \quad \Pi_j^{k-1}(r) = \big(\Pi_j^k\big)^r \bigg(1 + O_{\mathbb P}\Big(\frac 1 {X_k}\Big) \bigg).
\end{align*}
\end{lemma}
Uniformly in $0\le j < k < \tau_n$ is to be understood for example in i) in the sense that the family $\max_{0 \le k < \tau_n} |X_k^\beta-\big(\gamma(\tau_n-k)\big)^\beta|/(\tau_n-k)^{\beta +1/\alpha-1+\varepsilon}$ is tight and similarly for ii) and iii). In the sequel uniformity will continue to hold for all $O_{\mathbb P}$-statements depending on a parameter other than $n$. We will not point to it everywhere. 

\begin{proof}
Note that for $k \geq \tau_n - r$, the term in \eqref{eq:Pi} can be negative. However, since for such $k$ we have $1 \le X_k \le r$, the claims of the lemma then hold trivially. 
\medskip

 i) follows from Lemma 3.3 from \cite {DKW14}  and a Taylor expansion. In particular the argument uses the fact that $\frac {\tau_n} n$ converges in probability to $\frac 1 \gamma$, which is a consequence of \eqref{lm:tau}. The  equality in ii) is contained in Lemma 3.4 in \cite {DKW14}. For iii) note that using the fact that for $a_1, \dots, a_k, b_1, \dots, b_k$ positive
\begin{align*}
\Big|\prod_{i=1}^k a_i-\prod_{i=1}^k b_i \Big| &\leq \sum_{i=1}^k a_1\cdots a_{i-1} |a_i-b_i|b_{i+1}\cdots b_{k},
\end{align*}
it holds for $k < \tau_n-r$  
\begin{align*}
\Big|(\Pi_j^k)^r - \Pi_j^k(r)\Big| &\leq \sum_{i=j+1}^k \prod_{l=j+1}^{i-1}\Big(1-\frac 1 {X_l}\Big)^r\cdot\Big|\Big(1-\frac 1 {X_i}\Big)^r-\Big(1-\frac r {X_i}\Big)\Big| \cdot \prod_{l=i+1}^k\Big(1-\frac r {X_{l}}\Big)\\
& \leq \sum_{i=j+1}^k (\Pi_j^k)^r\cdot \Big(\frac{r+1}{r}\Big)^r \Big|\Big(1-\frac 1 {X_i}\Big)^r-\Big(1-\frac r {X_i}\Big)\Big|\\
& \leq (\Pi_j^k)^r \sum_{i=j+1}^k c\frac 1 {X_i^2}\\
& =O_{\mathbb P}\Big((\Pi_j^k)^r\cdot \frac 1 {X_k}\Big),
\end{align*}
where the last equality follows from the fact that
\begin{align}\label{sum_squares}
\sum_{i=1}^k \frac 1{X_i^2} \le \sum_{m\ge X_k} \frac 1{m^2} = O\Big(\frac 1{X_k} \Big). 
\end{align}
For the second equality in  iii) note that 
\[ \Pi_j^{k-1}(r) = \big(\Pi_j^{k-1}\big)^r \bigg(1 + O_{\mathbb P}\Big(\frac 1 {X_{k-1}}\Big) \bigg)= \big(\Pi_j^{k}\big)^r \bigg(1 + O_{\mathbb P}\Big(\frac 1 {X_{k-1}}\Big) \bigg)\Big(1+\frac 1 {X_{k}-1}\Big)^r\]
and since $\frac 1 {X_{k-1}}=O_{\mathbb P}\Big(\frac 1 {X_{k}}\Big)$  the claim follows.
\end{proof}

\begin{lemma}\label{lm:DKW_partB}
$\text{i) }$ 
For $0\leq k< \tau_n$
\begin{align*}
&  X_k = n - \gamma k - n^{1/\alpha} \mathcal{S}^{(n)}_{k/n}. 
\end{align*}
$\text{ii) }$ For $\beta \in \mathbb R$ it holds uniformly for $0\leq k< K_n$ with $K_n= \lfloor\frac n\gamma - n^{\delta}\rfloor$, $ 1/\alpha< \delta < 1$
\begin{align*}
X_k^{\beta} & = (n-\gamma k)^{\beta} \Big( 1 - \beta \frac{ n^{1/\alpha} \mathcal S^{(n)}_{k/n}}{n-\gamma k}
  + O_{\mathbb P}\Big(  \frac{n^{2/\alpha}}{(n-\gamma k)^{2} }\Big) \Big).
\end{align*}
\text{iii) } It holds uniformly for $0\leq j< k< K_n$ 
\[\Pi_0^k=\Big(\frac{n-\gamma k}{n}\Big)^{\frac 1 \gamma} \Big(1-(\alpha -1)\frac{n^{1/\alpha} \mathcal S_{k/n}^{(n)} }{n-\gamma k} + (\alpha -1)\sum_{j=1}^k \frac{\Delta_j - \gamma }{n-\gamma j}+ O_{\mathbb P}\Big(\frac {n^{2/\alpha}}{(n-\gamma k)^2}\Big)\Big). \]
\end{lemma}

\begin{proof}

i) This is just formula \eqref{eq:X_k with S_n}.

ii) Follows from i) and a Taylor expansion. Note that in the remainder the terms $(S_{k/n}^{(n)})^2$ are tight due to Lemma \ref{lm:S_n}.

iii) Using a Taylor expansion 
\begin{equation}\label{eq:Pi_taylor}
\Pi_0^k=\exp\Big(- \sum_{i=1}^k \frac1{X_i}+O\Big( \sum_{i=1}^k \frac1{X_i^2}\Big)\Big).
\end{equation}
By item ii)
\begin{equation}\label{eq:oneoverX}
\frac 1 {X_i} = \frac1{n-\gamma i} + \frac{n^{1/\alpha} \mathcal S^{(n)}_{i/n}}{(n-\gamma i)^2} 
+ O_{\mathbb P}\Big(\frac{n^{2/\alpha}}{(n-\gamma i)^3}\Big), \notag
\end{equation} 
hence 
\begin{align}\label{eq:Pi_sum1_1}
  \sum_{i=1}^k \frac1{X_i} 
  & = \sum_{i=1}^k \frac1{n-\gamma i} + \sum_{i=1}^k \frac{n^{1/\alpha} \mathcal S^{(n)}_{i/n}}{(n-\gamma i)^2} 
    + O_{\mathbb P} \Big( \frac {n^{2/\alpha}}{(n-\gamma k)^2}\Big) \\
  & = \frac1\gamma \log\frac{n}{n-\gamma k} +  \sum_{i=1}^k \frac{\sum_{j=1}^i (\Delta_j - \gamma)}{(n-\gamma i)^2}  + O_{\mathbb P} \Big( \frac {n^{2/\alpha}}{(n-\gamma k)^2}+ \frac {1}{n-\gamma k}\Big).
\end{align}
By interchanging the sums, using $\frac {n-\gamma k} {n^{2/\alpha}} =o(1)$ for the error term,
\begin{align}\label{eq:sumX}
  \sum_{i=1}^k \frac1{X_i}  
	& = \frac1\gamma \log\frac{n}{n-\gamma k} +  \frac 1 \gamma\sum_{j=1}^k (\Delta_j - \gamma) \Big( \frac{1 }{n-\gamma k} -\frac{1 }{n-\gamma j}  \Big) + O_{\mathbb P} \Big( \frac {n^{2/\alpha}}{(n-\gamma k)^2}\Big)\notag \\
	&= \frac1\gamma \log\frac{n}{n-\gamma k} +  \frac 1 \gamma \frac{n^{1/\alpha} \mathcal S_{k/n}^{(n)} }{n-\gamma k} - \frac 1 \gamma \sum_{j=1}^k \frac{\Delta_j - \gamma }{n-\gamma j}  + O_{\mathbb P} \Big( \frac {n^{2/\alpha}}{(n-\gamma k)^2}\Big).
\end{align}

Plugging this together with \eqref{sum_squares} into \eqref{eq:Pi_taylor} we obtain
\begin{align}\label{eq:Pi_taylor_2}
\Pi_0^k=\exp\Big(- \frac1\gamma \log\frac{n}{n-\gamma k} - \frac 1 \gamma \frac{n^{1/\alpha} \mathcal S_{k/n}^{(n)} }{n-\gamma k} + \frac 1 \gamma\sum_{j=1}^k \frac{\Delta_j - \gamma }{n-\gamma j}  + O_{\mathbb P} \Big( \frac {n^{2/\alpha}}{(n-\gamma k)^2}\Big)\Big).
\end{align}

Finally, for $k <  K_n$ 
\begin{align*}
\frac{n^{1/\alpha} \mathcal S_{k/n}^{(n)} }{n-\gamma k} + \sum_{j=1}^k \frac{\Delta_j - \gamma }{n-\gamma j}= O_{\mathbb P}\Big(\frac { n^{1/\alpha}}{n-\gamma k}\Big)= o_{\mathbb P}(1),
\end{align*}
and we end up  with 
\begin{align*}
\Pi_0^k=\Big(\frac{n-\gamma k}{n}\Big)^{\frac 1 \gamma} \Big(1-(\alpha -1)\frac{n^{1/\alpha} \mathcal S_{k/n}^{(n)} }{n-\gamma k} + (\alpha -1)\sum_{j=1}^k \frac{\Delta_j - \gamma }{n-\gamma j}+ O_{\mathbb P}\Big(\frac {n^{2/\alpha}}{(n-\gamma k)^2}\Big)\Big)
\end{align*}
which is the claim of the lemma.
\end{proof}

\subsection{Expectation given the block counting process}\label{Expectation}

In what follows we will focus on the case $r>1$. The case $r=1$ was handled in \cite{DKW14} (see Equation (3.14) therein for the analog of Lemma \ref{lm:cond_expectation_Z}).

For $m\in [r-1]$ and a vector $(r_1, \dots, r_m)$ such that $r_1,  \dots, r_m \in [r-1]$ and $r_1+ \dots+ r_m= r-1$ let
\begin{align}\label{r_hat}
   \widehat{r}_j := r - \sum_{i=1}^j r_i = r_{j+1}+r_{j+2}+\cdots+r_m+1
\end{align}
(read $\widehat{r}_0 = r, \,\widehat{r}_m=1$).

Note that in what follows we use the terms jumps of the block counting process and mergers in the coalescent tree equivalently.

\begin{lemma}\label{lm:cond_expectation_Z}
For the conditional expectation of the number $Z_{r,k}$ of $r$-blocks after $k$ jumps given $X$ it holds:
\begin{align*}
\E[Z_{r,k} \mid X]= X_k \sum_{(r_1, \dots, r_m)} \sum_{1\leq l_1 < \dots <l_m \leq k} \prod_{j=1}^m\Big(\mathds{1}_{\{\Delta_{l_j}=r_j\}} \cdot\widehat r_j\Big) \cdot \prod_{p=1}^{m} \left(  \frac 1 {X_{l_{p}}} \Pi_{l_{p-1}}^{l_{p}-1}\big(\widehat r_{p-1}\big) \right)  \cdot \Pi_{l_m}^{k},
\end{align*}
where the first sum is taken over all $m \in [r-1]$ and all $m$-tuples $(r_1, \dots, r_m)$ such that $r_1,  \dots, r_m \in [r-1]$ and $r_1+ \dots+ r_m= r-1$ and $l_0:=0$.
Furthermore
\begin{align}\label{eq:orderZ}
\E[Z_{r,k}| X]  = O_{\mathbb P} \Big( \tau_n^{1-\alpha} (\tau_n-k)^{\alpha} \Big).
\end{align}
\end{lemma}

\begin{proof} 
The following holds:
\begin{align}\label{eq:Z}
Z_{r,k} = \sum_{A\subset [n], |A|=r} \sum_{(r_1, \dots, r_m)} \sum_{1\leq l_1 < \dots <l_m \leq k} \mathds{1}_{E_{A,(r_1, \dots, r_m)}^{l_1, \dots, l_m}(k)}, 
\end{align}
where the second sum is taken over all $m \in [r-1]$ and all $m$-tuples $(r_1, \dots, r_m)$ such that $r_1,  \dots, r_m \in [r-1]$ and $r_1+ \dots+ r_m= r-1$. We denoted by \begin{align*}
E_{A,(r_1, \dots, r_m)}^{l_1, \dots, l_m}(k)
\end{align*}
the event that the block $A$ is formed through mergers occurring at times $T_{l_1}, \dots, T_{l_m}$ such that $(\Delta_{l_1}, \dots, \Delta_{l_m})= (r_1, \dots, r_m)$ and that the block still exists at level $k$, i.e. $A \in \Pi_k$. 

We define the following events: 
\begin{align}\label{def:A}
\mathcal A_{r,j}:=\{&\text{branches with leaves numbered $1, \dots, r$} \notag
\\
& \text{ are not involved in the first $j$ merger events} \}
\end{align}
and
\begin{align}\label{def:B}
\mathcal B_{r}:=\{& \text{branches with leaves numbered $r+1, \dots, X_0$} \notag
\\& \text{ are not involved in the first merger event} \}.
\end{align}
Taking the conditional expectation given $X$ in \eqref{eq:Z} and using exchangeability we obtain
\begin{align}\label{eq:cond_exp_Z_1}
\E[Z_{r,k} \mid X]= & \sum_{A\subset [n], |A|=r} \sum_{(r_1, \dots, r_m)} \sum_{1\leq l_1 < \dots <l_m \leq k} \mathbb E\Big(\mathds{1}_{E_{A,(r_1, \dots, r_m)}^{l_1, \dots, l_m}(k)}| X\Big) \notag
\\
& = \binom n r \sum_{(r_1, \dots, r_m)} \sum_{1\leq l_1 < \dots <l_m \leq k} \mathbb P\Big(E_{\{1,2, \dots,r\},(r_1, \dots, r_m)}^{l_1, \dots, l_m}(k)| X\Big).
\end{align}
It holds
\begin{align}\label{eq:prob_E}
\mathbb P\Big(E_{\{1,2, \dots,r\},(r_1, \dots, r_m)}^{l_1, \dots, l_m}(k)| X\Big) & =  \mathbb P_{X_0} (\mathcal A_{r, l_1-1}| X) \cdot \mathds{1}_{\{\Delta_{l_1}=r_1\}} \cdot \mathbb P_{X_{l_1-1}}(\mathcal B_r| X) \notag
\\ 
& \hspace{0.5cm} \cdot \mathbb P_{X_{l_1}} (\mathcal A_{r-r_1, l_2-l_1-1}| X) \cdot \mathds{1}_{\{\Delta_{l_2}=r_2\}} \cdot \mathbb P_{X_{l_2-1}}(\mathcal B_{r-r_1}| X) \notag \\
& \hspace{0.5cm} \dots \notag \\
&  \hspace{0.5cm} \cdot \mathbb P_{X_{l_m}} (\mathcal A_{r-r_1-\dots -r_m, k-l_m-\dots -l_1-1}| X).
\end{align}  
Note that the number of jumps taking part in the $k$-th jump is $\Delta_k+1$. Then (denoting by $(x)_j$ the $j$-th falling factorial of $x$)
\begin{align}\label{eq:prob_A}
\mathbb P_{X_0} (\mathcal A_{r,k} | X)& = \frac {\binom {X_0-r}{\Delta_1+1} \cdots \binom {X_{k-1}-r}{\Delta_k+1}}{\binom {X_0}{\Delta_1+1}\cdots \binom {X_{k-1}}{\Delta_k+1}} \notag\\
& = \frac {(X_0-r)\cdots (X_0-\Delta_1-r)}{X_0 \cdots (X_0-\Delta_1)}\cdots \frac {(X_{k-1}-r)\cdots (X_{k-1}-\Delta_k-r)}{X_{k-1}\cdots (X_{k-1}-\Delta_k)}\notag\\
&= \frac {(X_0-r)\cdots (X_1-r)}{X_0 \cdots X_1}\cdots \frac {(X_{k-1}-r)\cdots (X_k-r)}{X_{k-1}\cdots X_k} \notag\\
&= \frac {(X_0-r)\cdots (X_k-r+1)}{X_0 \cdots (X_k+1)}\prod_{j=1}^k \Big(1-\frac r {X_j}\Big) \notag\\
&= \frac {(X_0-r)_{X_0-X_k}}{(X_0)_{X_0- X_k}}\prod_{j=1}^k \Big(1-\frac r {X_j}\Big)
\end{align}
and
\begin{align}\label{eq:prob_B}
\mathbb P_{X_0} (\mathcal B_{r} | X)& = \frac {\binom {r}{\Delta_1+1}}{\binom {X_0}{\Delta_1+1}}  = \frac {r \cdots (r-\Delta_1)}{X_0 \cdots (X_0-\Delta_1)} =  \frac {r \cdots (r-\Delta_1+1)}{X_0 \cdots (X_1+1)}\cdot \frac{r-\Delta_1}{X_1}.
\end{align}
Plugging these in \eqref{eq:prob_E} we obtain
\begin{align*}
\mathbb P\Big(E_{\{1,2, \dots,r\},(r_1, \dots, r_m)}^{l_1, \dots, l_m}(k)| X\Big) &=  \prod_{p=1}^m\mathds{1}_{\{\Delta_{l_p}=r_p\}} \cdot \frac {(X_0-r)_{X_0-X_{l_1-1}}}{(X_0)_{X_0-X_{l_1-1}}}\prod_{j=1}^{l_1-1} \Big(1-\frac r {X_j}\Big) \\
& \hspace{1cm} \cdot \frac {r \cdots (r-r_1+1)}{X_{l_1-1} \cdots (X_{l_1}+1)}\cdot \frac{r-r_1}{X_{l_1}} \\
& \cdot \frac {(X_{l_1}-r+r_1)_{X_{l_1}-X_{l_2-1}}}{(X_{l_1})_{X_{l_1}-X_{l_2-1}}}\prod_{j=l_1+1}^{l_2-1} \Big(1-\frac {r-r_1} {X_j}\Big) \\
& \hspace{1cm} \cdot \frac {(r-r_1) \cdots (r-r_1-r_2+1)}{X_{l_2-1} \cdots (X_{l_2}+1)}\cdot \frac{r-r_1-r_2}{X_{l_2}} \\
& \cdots\\
& \cdot \frac{(X_{l_m}-1)_{X_{l_m}-X_k}}{(X_{l_m})_{X_{l_m}-X_k}}  \prod_{j=l_m+1}^{k} \Big(1-\frac {1} {X_j}\Big).
\end{align*}
Note that since $X_{l_i}-r+\Delta_{l_i}=X_{l_i-1}-r$ for $i=1, \dots, m$, the product of the falling factorials in the nominator is equal to $(X_0-r) \cdots X_k$.  We thus obtain that
\begin{align}\label{eq:prob_E_final}
\mathbb P\Big(E_{\{1,2, \dots,r\},(r_1, \dots, r_m)}^{l_1, \dots, l_m}(k)| X\Big) = & \, \frac{(X_0-r) \cdots X_k}{X_0\cdots (X_k+1)} \cdot  r!  \notag \\
& \cdot \prod_{p=1}^m \Big(\mathds{1}_{\{\Delta_{l_p}=r_p\}}\frac{\widehat r_p}{X_{l_p}} \cdot \prod_{j=l_{p-1}+1}^{l_p-1} \Big(1-\frac {\widehat  r_{p-1}} {X_j}\Big) \Big) \notag \\
& \cdot \prod_{j=l_m+1}^{k} \Big(1-\frac {1} {X_j}\Big).
\end{align}
Since $X_0=n$ \eqref{eq:prob_E_final} together with \eqref{eq:cond_exp_Z_1} gives the first claim.

From \eqref{eq:prob_E_final} and Lemma \ref{lm:DKW} we obtain
\begin{align*}
& \E[Z_{r,k}| X] 
\\
& = O_{\mathbb P} \Big(\sum_{(r_1, \dots,r_m)} \sum_{1\leq l_1<\dots<l_m\leq k} \prod_{p=1}^m X_k \Big( \Big(\Pi_{l_{p-1}}^{l_p-1}\Big)^{\widehat  r_{p-1}}\cdot \frac 1 {X_{l_p}}\Big) \Pi_{l_m}^k \Big)\notag \\
& = O_{\mathbb P} \Big(\gamma(\tau_n-k)\sum_{(r_1, \dots,r_m)} \sum_{1\leq l_1<\dots<l_m\leq k} \prod_{p=1}^m  \Big( \Big(\frac{\tau_n-l_p-1}{\tau_n-l_{p-1}}\Big)^{(\alpha-1)\widehat  r_{p-1}} \frac 1 {\gamma(\tau_n-l_p)}\Big)\Big(\frac{\tau_n-k}{\tau_n-l_m}\Big)^{\alpha-1} \Big)\notag \\
& = O_{\mathbb P} \Big((\tau_n-k)\sum_{(r_1, \dots,r_m)} \sum_{1\leq l_1<\dots<l_m\leq k} \prod_{p=1}^m  \Big(\Big(\frac{\tau_n-l_p}{\tau_n}\Big)^{(\alpha-1)r_p}\cdot \frac 1 {\tau_n-l_p}\Big)\Big(\frac{\tau_n-k}{\tau_n}\Big)^{\alpha-1} \Big)\notag 
\end{align*}
Note that by allowing $l_1, \dots, l_m$ also to be equal we only enlarge the respective sum and therefore
\begin{align*}
\E[Z_{r,k}| X]  & = O_{\mathbb P} \Big( (\tau_n-k) \Big(\frac{\tau_n-k}{\tau_n}\Big)^{\alpha-1}  \sum_{(r_1, \dots,r_m)}  \prod_{p=1}^m  \Big( \sum_{l=1}^k \Big(\frac{\tau_n-l}{\tau_n}\Big)^{(\alpha-1)r_p}\cdot \frac 1 {\tau_n-l}\Big)\Big)\notag \\
& =  O_{\mathbb P} \Big( (\tau_n-k)^{\alpha} \tau_n^{1-\alpha}   \sum_{(r_1, \dots,r_m)} \tau_n^{(1-\alpha)\sum_{i=1}^m r_i } \prod_{p=1}^m  \Big( \sum_{l=1}^k \Big((\tau_n-l)^{(\alpha-1)r_p-1}\Big)\Big) \Big)\notag \\
& =  O_{\mathbb P} \Big( (\tau_n-k)^{\alpha} \tau_n^{r(1-\alpha)} \sum_{(r_1, \dots,r_m)} \prod_{p=1}^m  \Big( \tau_n^{(\alpha-1)r_p}-(\tau_n-k)^{(\alpha-1)r_p}\Big)\Big)\notag \\
& = O_{\mathbb P} \Big( \tau_n^{1-\alpha} (\tau_n-k)^{\alpha} \Big).
\end{align*}
This gives the second claim.
\end{proof}

\subsection{Quantifying the approximations in \eqref{eq:approximations}}\label{Quantifying}

We start with the second approximation in \eqref{eq:approximations}.
Let
\[\bar \ell_r:=\sum_{k=0}^{\tau_n-1} \frac{ \E[Z_{r,k} | X] }{ \lambda_{X_k}}.\]  

\begin{lemma}
It holds that
\[\sum_{k=0}^{\tau_n-1} \frac{ Z_{r,k} }{ \lambda_{X_k}} = \bar \ell_r + o_{\mathbb P}(n^{1-\al+\frac 1 \al}).\]
\end{lemma}

\begin{proof}
For $r=1$ we refer to Lemma 3.2 in \cite{DKW14} whose proof we adapt here for the case $r\geq 2$.
Define
\[H_k=H_k(r):= \text{number of branches of order $r$ that participate in the $k$-th jump of $X$} \]
(recall that a branch of order $r$ refers to a block with $r$ elements) and note that
\begin{equation}\label{eq:hyp}
\mathcal L (H_k | X, Z_{r, k-1}) = \text{Hyp} (X_{k-1}, Z_{r, k-1}, \D_k +1),
\end{equation}
where $\text{Hyp}$ refers to the hypergeometric distribution, i.e.\ given $X$ and $Z_{r, k-1}$, $H_k$
is distributed as the number of red balls drawn when drawing $\D_k +1$ times without replacement
from an urn which contains $Z_{r, k-1}$ red balls and $X_{k-1}-Z_{r, k-1}$ non-red balls. 
Define also
\[\mathscr A_k=\mathscr A_k(r):= \{\text{a branch of order $r$ is formed through the $k$-th merger}\}.\]
Then for $k\geq 1$
\[Z_{r,k}= Z_{r,k-1}- H_k + \mathds 1_{\mathscr A_k}\]
and
\begin{align*}
\E [Z_{r,k}|X,Z_{r,k-1}]&= Z_{r,k-1}- \frac{Z_{r,k-1}}{X_{k-1}} (\D_k +1)+ \Prob {\mathscr A_k |X,Z_{r,k-1}}
\\
&= Z_{r,k-1} \frac{X_k-1}{X_{k-1}} + \Prob{ \mathscr A_k |X,Z_{r,k-1}}.
\end{align*}
Thus
\begin{align}
\E [Z_{r,k}|X]= \E[Z_{r,k-1}|X] \frac{X_k-1}{X_{k-1}} + \Prob{\mathscr A_k |X}.
\end{align}
Note also that
\[Z_{r,k}= Z_{r,k-1}\frac{X_k-1}{X_{k-1}} +Z_{r,k-1} \frac{\D_k +1}{X_{k-1}} - H_k + \mathds 1_{\mathscr A_k}\]
and by denoting 
\[G_k:=Z_{r,k} - \E[Z_{r,k} | X]\]
we obtain
\begin{equation}\label{eq:Gk}
\frac {G_k}{X_k} =\frac{G_{k-1}}{X_{k-1}} \Big( 1 - \frac{1}{X_k} \Big) - \frac{\widetilde H_k}{X_k}, 
\end{equation}
with
\[\widetilde H_k: = H_k - Z_{r,k-1} \frac{\D_k +1}{X_{k-1}} - \mathds 1_{\mathscr A_k}+ \Prob{\mathscr A_k |X}.\]
Recall \eqref{eq:Pi}. Then \eqref{eq:Gk} is equivalent to
\[\frac {G_k}{X_k\Pi_0^k} =\frac{G_{k-1}}{X_{k-1}\Pi_0^{k-1}} - \frac{\widetilde H_k}{X_k\Pi_0^k}. \]
We iterate this relation and since $G_0=0$ we obtain
\[G_k= - X_k \Pi_0^k \sum_{i=1}^k \frac {\widetilde H_i}{X_i \Pi_0^i}.\]
It follows that
\[ \bar\ell_r - \sum_{k=0}^{\tau_n-1} \frac{ Z_{r,k} }{ \lambda_{X_k}} =\sum_{k=0}^{\tau_n-1} \frac{X_k\Pi_0^k}{\lambda_{X_k}} \sum_{i=1}^k \frac{\widetilde H_i}{X_i\Pi_0^i} = \sum_{i=1}^{\tau_n-1} \frac{\widetilde H_i}{X_i} \sum_{k=i}^{\tau_n-1} \frac{X_k\Pi_i^k}{\lambda_{X_k}}.  \]
We denote by
\[Z_i:=(Z_{1, i}, Z_{2,i}, \dots, Z_{r, i})\]
the vector recording the numbers of branches of orders $1,2,\dots, r$ present in the tree after $i$ mergers. We further write $\tilde H_i=J_i +I_i $ where
\[ J_i:= H_i - Z_{r,i-1} \frac{\D_i +1}{X_{i-1}} - \mathds 1_{\mathscr A_i}+ \Prob{\mathscr A_i |X,Z_{i-1}} \quad \text{ and } \quad I_i:=-\Prob{\mathscr A_i |X,Z_{i-1}}+\Prob{\mathscr A_i |X}.\]
Thus
\begin{align}\label{eq:HJI}
\bar\ell_r - \sum_{k=0}^{\tau_n-1} \frac{ Z_{r,k} }{ \lambda_{X_k}}= \sum_{i=1}^{\tau_n-1} \frac{J_i}{X_i} \sum_{k=i}^{\tau_n-1} \frac{X_k\Pi_i^k}{\lambda_{X_k}} +  \sum_{i=1}^{\tau_n-1} \frac{I_i}{X_i} \sum_{k=i}^{\tau_n-1} \frac{X_k\Pi_i^k}{\lambda_{X_k}}=:D_1+D_2.  
\end{align}
We start by bounding $D_1$. Note that given $X$ and $Z_{i-1} $, the random variable $J_i$ is centered and therefore $(J_i)$ forms a sequence of martingale differences with respect to the filtration generated by $(X, Z_{0}, Z_{1}, \dots Z_{i-1})$. Thus $J_1,  J_2, \dots$ are uncorrelated given $X$. Moreover, from \eqref{eq:hyp}, since $J_i $ is centered
\[\E[J_i^2 | X ] = \E\big[ \E[J_i^2 | X, Z_{i-1} ] \big| X \big] \leq 2\Big((\D_i +1 )  \frac{\E[Z_{r,i-1}| X]} {X_{i-1}}+ 1\Big)\]
and
\begin{align}\label{eq:D1}
 \E[ D_1^2 | X ] & = \sum_{i=1}^{\tau_n-1} \frac{\E[ J_i^2|X]}{X_i^2} \Big(\sum_{k=i}^{\tau_n-1} \frac{X_k\Pi_i^k}{\lambda_{X_k}}  \Big)^2 \notag
\\
& \leq \sum_{i=1}^{\tau_n-1} \frac{2\Big((\D_i +1 )  \frac{\E[Z_{r,i-1}| X]} {X_{i-1}}+ 1\Big)}{X_i^2} \Big(\sum_{k=i}^{\tau_n-1} \frac{X_k\Pi_i^k}{\lambda_{X_k}}\Big)^2 
\end{align}

Using \eqref{eq:orderZ}, \eqref{eq:lambda_m} and Lemma \ref{lm:DKW} in \eqref{eq:D1} we obtain
\begin{align*}
 \E[ D_1^2 | X ] & =  O_{\mathbb P} \Big( \sum_{i=1}^{\tau_n-1} \frac{(\D_i +1 ) \Big( \frac {\tau_n-i} {\tau_n-1}\Big)^{\al-1}+ 1}{(\tau_n-i)^2} \Big((\tau_n-i)^{2(1-\al)} (\tau_n-i)^2\Big) \Big) \\
& =  O_{\mathbb P} \Big( \tau_n^{1-\al} \sum_{i=1}^{\tau_n-1} (\D_i +1 ) (\tau_n-i)^{1-\al} +\sum_{i=1}^{\tau_n-1} (\tau_n-i)^{2-2\al}\Big). 
\end{align*}

Again from Lemma \ref{lm:DKW} $X_i$ is of order $\tau_n-i$ uniformly in $0\leq i<\tau_n$. Since $X_{i-1}-X_i +1 \leq 2(X_{i-1}-X_i)$ it follows that
\begin{align*}
 \E[ D_1^2 | X ] & =  O_{\mathbb P} \Big( \tau_n^{1-\al} \sum_{i=1}^{\tau_n} (\D_i +1 ) X_i^{1-\al} +\sum_{i=1}^{\tau_n} (\tau_n-i)^{2-2\al}\Big). 
\end{align*}
Therefore, using (observe that $x \mapsto x^{1-\alpha}$ is decreasing)
\[\sum_{i=1}^{\tau_n-1}  X_{i}^{1-\al}(X_{i-1}-X_i) \leq \int_{X_{\tau_n-1}}^{X_0}x^{1-\al}dx\leq \frac 1 {2-\al} X_0^{2-\al}\] 
we obtain that 
\begin{align*}
 \E[ D_1^2 | X ] & =  O_{\mathbb P} \Big( n^{3-2\al} \Big).
\end{align*}
Since $\frac 3 2 - \al < 1-\al + \frac 1 \al$ it holds by Cauchy-Schwarz that
\begin{align}\label{eq:D1_final}
 \E[ |D_1| \,\,| X ] & =  o_{\mathbb P} \Big( n^{1-\al + \frac 1 \al} \Big).
\end{align}

We now bound $D_2$, the second term on the left hand side of \eqref{eq:HJI}. Note that
\begin{align*}
\E[\1_{\mathscr A_k}| X, Z_{k-1} ] & =\sum_{u=2}^{r} \1_{\Delta_{k}=u-1} \sum_{a_1, \dots, a_r} \frac{\prod_{i=1}^{r-1} \binom{Z_{i,k-1}}{a_i}} {\binom {X_{k-1}}{u}}
\\
& =\sum_{u=2}^{r} \1_{\Delta_{k}=u-1} \sum_{a_1, \dots, a_r} c_{u, a_1, \dots, a_r}\frac{\prod_{i=1}^{r-1} (Z_{i,k-1})_{(a_i)}} {(X_{k-1})_u},
\end{align*}
where the second sum is taken over $a_1, \dots, a_r$ having the properties: $a_j \in \{0, \dots, Z_{j, k-1}\}$, $\sum_{j=1}^r ja_j=r$ and $\sum_{j=1}^r a_j=u$.  $c_{u, a_1, \dots, a_r}>0$ is a constant. (A branch of order $r$ is formed through a merger of size $u$, involved are $a_j$ branches of order $j$.)

Now since the sums and the numbers $a_1, \dots, a_r$ are bounded, in order to bound $|I_k|=\Big| \E[\1_{\mathscr A_k}| X, Z_{k-1}] - \E[\1_{\mathscr A_k}| X]\Big|$ it suffices to bound, for $u \in 2, \dots, r$, terms of the form 
\begin{align}\label{def:Y}
Y_{u,k}:=\Big|\frac { \prod_{i=1}^u Z_{r_i, k}}{X_{k-1}^u} - \frac {\E[\prod_{i=1}^u Z_{r_i, k} | X]}{X_{k-1}^u}\Big|,
\end{align}
with $1\leq r_i\leq r$.  Note that 
\begin{align*}
\E \Big[&\Big| \prod_{i=1}^u Z_{r_i, k}- \E[ \prod_{i=1}^u Z_{r_i, k} | X]\Big| \, |X \Big] \\
&\leq \E \Big[\Big| \prod_{i=1}^uZ_{r_i, k}- \prod_{i=1}^u\E[Z_{r_i, k} | X]\Big| \, |X \Big]+\E \Big[\Big| \prod_{i=1}^u \E[Z_{r_i, k} | X] - \E[ \prod_{i=1}^u Z_{r_i, k} | X]\Big| \, |X \Big] \\
& = \E \Big[\Big| \prod_{i=1}^uZ_{r_i, k}- \prod_{i=1}^u\E[Z_{r_i, k} | X]\Big| \, |X \Big]+ \Big| \prod_{i=1}^u \E[Z_{r_i, k} | X] - \E[ \prod_{i=1}^u Z_{r_i, k} | X]\Big|
\\
& = \E \Big[\Big| \prod_{i=1}^uZ_{r_i, k}- \prod_{i=1}^u\E[Z_{r_i, k} | X]\Big| \, |X \Big]+ \Big| \E\Big[\prod_{i=1}^u \E[Z_{r_i, k} | X] - \prod_{i=1}^u Z_{r_i, k} | X\Big]\Big|
\\
& \leq 2 \E \Big[\Big| \prod_{i=1}^uZ_{r_i, k}- \prod_{i=1}^u\E[Z_{r_i, k} | X]\Big| \, |X \Big].
\end{align*}
Now using the fact that for positive numbers $a_1, \dots, a_k, b_1, \dots, b_k$  it holds that
\begin{align*}
\Big|\prod_{j=1}^k a_j-\prod_{j=1}^k b_j \Big| &\leq \sum_{j=1}^k a_1\cdots a_{j-1} |a_j-b_j|b_{j+1}\cdots b_{k},
\end{align*}
and that $Z_{r_i, k}\leq X_{k-1}$ we further obtain that
\begin{align*}
\E \Big[\Big| \prod_{i=1}^u Z_{r_i, k}- \prod_{i=1}^u\E[Z_{r_i, k} | X]\Big| \, |X \Big] & \leq \E \Big[ \sum_{i=1}^u | Z_{r_i, k}- \E[Z_{r_i, k} | X] | \cdot X_{k-1}^{u-1} \, |X \Big] \\
& = X_{k-1}^{u-1} \sum_{i=1}^u \E \Big[| Z_{r_i, k}- \E[Z_{r_i, k}  | X] | \, |X \Big].
\end{align*}
Thus 
\begin{align}\label{eq:Y}
Y_{u,k}\leq \frac 2 {X_{k-1}} \sum_{i=1}^u \E \Big[| Z_{r_i, k}- \E[Z_{r_i, k} | X] | \, |X \Big]\leq \frac 2 {X_{k-1}} \sum_{i=1}^u  \E \Big[( Z_{r_i, k}- \E[Z_{r_i, k} | X] )^2\, |X \Big]^{1/2}.
\end{align}
We now show that for any $r\in [n]$
\begin{align}\label{ineq:Z}
 \E \Big[( Z_{r, k}- \E[Z_{r, k} | X] )^2\, |X \Big] \leq c \E \Big[ Z_{r, k}\, |X \Big],
\end{align}
for $c>0$ constant.
It holds by \eqref{eq:Z} that
\begin{align}\label{eq:Z_squared}
\E\Big[Z^2_{r, k}\, |X \Big] &= \E\Big[\Big(\sum_{A\subset [n], |A|=r} \sum_{(r_1, \dots, r_m)} \sum_{1\leq l_1 < \dots <l_m \leq k} \mathds{1}_{E_{A,(r_1, \dots, r_m)}^{l_1, \dots, l_m}(k)} \, \Big)^2 |X \Big] \notag \\
& = \E\Big[\sum_{{A'\subset [n], |A'|=r} \atop {A''\subset [n], |A''|=r}} \sum_{{(r'_1, \dots, r'_{m'})} \atop {(r''_1, \dots, r''_{m''})}} \sum_{{1\leq l'_1 < \dots <l'_{m'} \leq k} \atop {1\leq l''_1 < \dots <l''_{m''} \leq k}} \mathds{1}_{E_{A',(r'_1, \dots, r'_{m'})}^{l'_1, \dots, l'_{m'}}(k) \cap E_{A'',(r''_1, \dots, r''_{m''})}^{l''_1, \dots, l''_{m''}}(k)} \, |X \Big].
\end{align}
Note that $E_{A',(r'_1, \dots, r'_{m'})}^{l'_1, \dots, l'_{m'}}(k) \cap E_{A'',(r''_1, \dots, r''_{m''})}^{l''_1, \dots, l''_{m''}}(k) =\emptyset$ if $A' \cap A''\neq\emptyset$ and $A' \neq A''$. Let now $A' \cap A''=\emptyset$ and let $\{l'_1, \dots, l'_{m'}\} \cap \{l''_1, \dots, l''_{m''}\}=\emptyset$. Denote $m:=m'+m''$ and by $0=l_0 <l_1<\dots <l_m<k$ the ordered set of levels $l'_1, \dots, l'_{m'}, l''_1, \dots, l''_{m''}$. Let also $r_1, \dots, r_m$ be the sequence of jumps the $r'_1, \dots, r'_{m'}, r''_1, \dots, r''_{m''}$ rearranged to correspond to $l_1\dots l_m$. It holds (recall \eqref{def:A} and \eqref{def:B})
\begin{align*}
\mathbb P \Big( &E_{A',(r'_1, \dots, r'_{m'})}^{l'_1, \dots, l'_{m'}}(k) \cap E_{A'',(r''_1, \dots, r''_{m''})}^{l''_1, \dots, l''_{m''}}(k)  \, | X \Big)\\
 & = \prod_{i=1}^{m} \mathbb P_{X_{l_{i-1}}} \Big( \mathcal A_{2r - r_1-\dots - r_{i-1}, l_i-l_{i-1}-1}\Big) \cdot \1_{\Delta_{l_{i}}=r_i}\\
&\hspace{1cm} \cdot \prod_{i=1}^{m'} \mathbb P_{X_{l'_{i}-1}} \Big( \mathcal B_{r - r'_1-\dots - r'_{i-1}}\Big) \cdot \prod_{i=1}^{m''} \mathbb P_{X_{l''_{i}-1}} \Big( \mathcal B_{r - r''_1-\dots - r''_{i-1}}\Big) \cdot \mathbb P_{X_{l_m}} \Big( \mathcal A_{2, k-l_m}\Big).
\end{align*}
Using \eqref{eq:prob_A} and \eqref{eq:prob_B} and a similar argument as in the proof of Lemma \ref{lm:cond_expectation_Z} we further obtain that
\begin{align*}
\mathbb P \Big( E_{A',(r'_1, \dots, r'_{m'})}^{l'_1, \dots, l'_{m'}}(k) &\cap E_{A'',(r''_1, \dots, r''_{m''})}^{l''_1, \dots, l''_{m''}}(k)  \, | X \Big) \\&=   \, \frac {(X_0-2r)\dots X_k(X_k-1)}{X_0\dots (X_k+1)}\cdot r'! \,r''! \\
& \qquad \cdot \prod_{i=1}^m \Big(\1_{\Delta_{l_{i}}=r_i} \cdot\prod_{j=l_{i-1}+1}^{l_i-1} \Big(1- \frac {2r-r_1-\dots-r_{i-1}}{X_j}\Big)\Big)\\
& \qquad \cdot \prod_{i=1}^{m'} \frac {r-r'_1-\dots - r'_i}{X_{l'_i}} \cdot \prod_{i=1}^{m''} \frac {r-r''_1-\dots - r''_i}{X_{l''_i}} \cdot \prod_{j=l_m+1}^{k} \Big(1-\frac {2}{X_{j}}\Big)
\end{align*}
Using \eqref{eq:prob_E_final} we obtain for a constant $c>0$ 
that

\begin{align*}
\mathbb P \Big( E_{A',(r'_1, \dots, r'_{m'})}^{l'_1, \dots, l'_{m'}}(k) &\cap E_{A'',(r''_1, \dots, r''_{m''})}^{l''_1, \dots, l''_{m''}} (k) \, | X \Big) 
\\
& \leq 
\mathbb P \Big( E_{A',(r'_1, \dots, r'_{m'})}^{l'_1, \dots, l'_{m'}}(k)  \, | X \Big)\cdot \mathbb P \Big( E_{A'',(r''_1, \dots, r''_{m''})}^{l''_1, \dots, l''_{m''}}(k)  \, | X \Big) \cdot \Big(1+\frac c {X_k}\Big).
\end{align*}
Therefore, by \eqref{eq:Z_squared} and using for the last inequality the fact that $\E[Z_{r, k}\, |X ] \le X_k $ we further get that
\begin{align*}
\E\Big[Z^2_{r, k}\, |X \Big] & \leq \E\Big[\sum_{{A\subset [n], |A|=r} } \sum_{{(r_1, \dots, r_m)} } \sum_{{1\leq l_1 < \dots <l_m \leq k} } \mathds{1}_{E_{A,(r_1, \dots, r_m)}^{l_1, \dots, l_m}(k)} \, |X \Big] \\
& \hspace{0.1cm} + \sum_{{A'\cap A''=\emptyset}} \sum_{{(r'_1, \dots, r'_{m'})} \atop {(r''_1, \dots, r''_{m''})}} \sum_{{1\leq l'_1 < \dots <l'_{m'} \leq k} \atop {1\leq l''_1 < \dots <l''_{m''} \leq k}} \mathbb P \Big( E_{A',(r'_1, \dots, r'_{m'})}^{l'_1, \dots, l'_{m'}}(k)  \, | X \Big) \mathbb P \Big( E_{A'',(r''_1, \dots, r''_{m''})}^{l''_1, \dots, l''_{m''}}(k)  \, | X \Big) \Big(1+\frac c {X_k}\Big)
\\
&\leq \mathbb E[Z_{r,k} | X]+\Big(1+\frac c {X_k}\Big) \Big(\sum_{{A'\subset [n], |A'|=r} \atop {A''\subset [n], |A''|=r}} \sum_{{(r'_1, \dots, r'_{m'})} \atop {(r''_1, \dots, r''_{m''})}} \sum_{{1\leq l'_1 < \dots <l'_{m'} \leq k} \atop {1\leq l''_1 < \dots <l''_{m''} \leq k}} \mathbb P \Big( E_{A,(r_1, \dots, r_m)}^{l_1, \dots, l_m}(k) \Big)\Big)^2
\\
& =\mathbb E[Z_{r,k} | X]+ \Big(1+\frac c {X_k}\Big) \E[Z_{r, k}\, |X ]^2 \\
& = \mathbb E[Z_{r,k} | X]^2+ \E[Z_{r, k}\, |X ] \Big(1+c\frac  {\E[Z_{r, k}\, |X]}{X_k}\Big) 
\\
& \leq \mathbb E[Z_{r,k} | X]^2 + (1+c) \mathbb E[Z_{r,k} | X],
\end{align*}
which proves \eqref{ineq:Z}. Thus \eqref{eq:Y} becomes
\begin{align}\label{eq:Y2}
Y_{u,k}\leq \frac 2 {X_{k-1}} \sum_{i=1}^u \E \Big[| Z_{r_i, k}- \E[Z_{r_i, k} | X] \, |X \Big]\leq \frac c {X_{k-1}} \sum_{i=1}^u  \E [ Z_{r_i, k} \, |X ]^{1/2}.
\end{align}
Now, using the observation before formula \eqref{def:Y} together with \eqref{eq:Y2}, \eqref{eq:lambda_m} and Lemma \ref{lm:DKW} we obtain 
\begin{align*}
\E[|D_2| \, \, | X] &\leq\sum_{i=1}^{\tau_n-1} \frac{1}{X_i} \sum_{k=i}^{\tau_n-1} \frac{X_k\Pi_i^k}{\lambda_{X_k}} \cdot \E\big[|I_i| \,\, | X\big]
\\
& = O_{\mathbb P}\Big( \sum_{i=1}^{\tau_n-1} \frac{1}{X_i^{\al+1}} (\tau_n-i) \sum_{s=1}^{r-1} \E \Big[ Z_{s, i} \, |X \Big]^{1/2}\Big).
\end{align*}
By \eqref{eq:orderZ} and \eqref{lm:tau}
\begin{align*}
\E[|D_2| \, \, | X] & = O_{\mathbb P}\Big( n^{1/2}\sum_{i=1}^{\tau_n-1} \frac{1}{X_i^{\al+1}} (\tau_n-i) \Big(\frac {X_i } n\Big)^{\alpha /2}\Big)\\
& = O_{\mathbb P} \Big(n^{1/2-\alpha/2} \sum_{i=1}^{\tau_n}  (n-\gamma i)^{-\alpha /2}\Big) = O_{\mathbb P} \Big( n^{1/2-\alpha/2} n^{1-\alpha /2}\Big)= O_{\mathbb P} \Big( n^{\frac 3 2 - \alpha}\Big).
\end{align*}
Since $\frac 3 2 - \alpha <1-\alpha+\frac 1 \alpha$, this together with \eqref{eq:HJI} and \eqref{eq:D1_final} proves the lemma.
\end{proof}

\medskip

We proceed to evaluate the length $\bar \ell_r =  \sum_{k=0}^{\tau_n-1}  \frac {\mathbb E [Z_{r,k} | X]}{\lambda_{X_k}}$.
Note from Lemma \ref{lm:cond_expectation_Z},
Lemma \ref{lm:DKW} iii) and \eqref{eq:lambda_m} that
\begin{align*}
\bar \ell_r & = \al \Gamma(\al) \sum_{(r_1, \dots, r_m)} \sum_{1\leq l_1<\dots<l_m\leq k<\tau_n} X_k^{1-\al} \Big(1 + O_\mathbb P\Big(\frac 1 {X_k}\Big)\Big) \cdot \Pi_{l_m}^{k}  \prod_{p=1}^m \widehat r_p  \cdot \mathds 1_{\{\Delta_{l_p}=r_p\}} \\
& \hspace{2cm} \cdot \prod_{p=1}^m \Big( \Big( \Pi_{l_{p-1}}^{l_p-1} \Big)^{\widehat r_{p-1}} \Big(1+ O_{\mathbb P}\Big(\frac {1} {X_{l_{p}-1}}\Big)\Big) \frac 1 {X_{l_p}},
\end{align*} 
where the first sum is taken over all $m \in [r-1]$ and all $m$-tuples $(r_1, \dots, r_m)$ such that $r_1,  \dots, r_m \in [r-1]$ and $r_1+ \dots+ r_m= r-1$ and $l_0:=0$. Note that by redistributing the factors in the products
\begin{align*}
\prod_{p=1}^m  \Big( \Pi_{l_{p-1}}^{l_p-1} \Big)^{\widehat r_{p-1}} 
& =
\prod_{p=1}^m \Big( \Big( \Pi_0^{l_{p-1}} \Big)^{ r_{p}}\frac {(\Pi_0^{l_{p}})^{r_{p+1}+\dots+r_m+1}}{(\Pi_0^{l_{p-1}})^{r_{p}+\dots+r_m+1}} \Big(\frac 1 {1-\frac 1 {X_{l_p}}} \Big)^{r_{p+1}+\dots+r_m+1}
\\
& =
\prod_{p=1}^m  \Big( \Pi_0^{l_{p-1}} \Big)^{ r_{p}}\Pi_0^{l_m} \Big(1+ O_\mathbb P \Big(\frac 1  {X_{l_k}} \Big)\Big)
\end{align*}  and using Lemma \ref{lm:DKW} we obtain
\begin{align}\label{eq:l1bar_0}
\bar \ell_r & = \al \Gamma(\al) \sum_{(r_1, \dots, r_m)} \sum_{1\leq l_1<\dots<l_m\leq k<\tau_n} X_k^{1-\al} \Pi_{0}^{k} \prod_{p=1}^m\widehat r_{p}  \mathds 1_{\{\Delta_{l_p}=r_p\}} \cdot \prod_{p=1}^m \Big( \Pi_{0}^{l_p-1} \Big)^{r_p} \frac 1 {X_{l_p}} + R
\end{align}
with
\begin{align*}
R & = O_{\mathbb P}\Bigg(\sum_{(r_1, \dots, r_m)} \sum_{1\leq l_1<\dots<l_m\leq k<\tau_n} X_k^{1-\al} \frac 1 {X_k} \cdot \Pi_{0}^{k} \cdot  \prod_{p=1}^m \Big( \Pi_{l_{p-1}}^{l_p-1} \Big)^{\widehat r_{p-1}} \frac 1 {X_{l_p}}\mathds 1_{\{\Delta_{l_p}=r_p\}} \Bigg)
\\
& = O_{\mathbb P}\Big( \sum_{k=1}^{\tau_n-1} \frac {\mathbb E[Z_{r,k} \mid X]} {X_k^{\alpha+1}}\Big)
\\
& = O_{\mathbb P}\Big( \sum_{k=1}^{\tau_n-1} \frac { \tau_n^{1-\alpha} (\tau_n-k)^{\alpha}} {(\tau_n-k)^{\alpha+1}}\Big)
\\
& =  o_{\mathbb P}\Big( {n}^{1-\alpha+1/\alpha}\Big).
\end{align*}

Recall the coupling between of the jumps of the block counting process and the sequence $V_1, V_2, \dots$ given in Section \ref{Preparations}.
We show that when replacing the indicators in \eqref{eq:l1bar_0} by the corresponding probability weights of the random variable $V$ we make an error of just $o_\mathbb P (n^{1-\alpha+1/\alpha})$. 

\begin{lemma}\label{lm:indicators_new}
For $r\geq 2$ we have
\begin{align*}
\bar \ell_r= \al \Gamma(\al) & \sum_{(r_1, \dots, r_m)} \bigg(  \prod_{p=1}^m \widehat r_p \mathbb P( V=r_p) \bigg) \cdot  \sum_{1\leq l_1<\dots<l_m\leq k<\tau_n} X_k^{1-\al} \Pi_0^{k} \prod_{p=1}^m \frac 1 {X_{l_p}} \Big( \Pi_{0}^{l_p}\Big) ^{r_p} \notag \\
&+ o_{\mathbb P} \Big( n^{1/\al+1-\al} \Big).
\end{align*}
\end{lemma}

Note that the analogous formula for $r=1$ reads 
(compare the proof of Lemma~2.3 in \cite{DKW14})
\begin{align}
  \label{barell1}
\bar \ell_1 = \al \Gamma(\al) \sum_{k=1}^{\tau_n-1} 
X_k^{1-\al} \cdot \Pi_{0}^{k} + 
o_{\mathbb P} \Big( n^{1/\al+1-\al} \Big).
\end{align}
This is consistent with the formula from Lemma~\ref{lm:indicators_new} when interpreting empty sums as inexistent and empty products as $1$ and can be proved completely analogously to the case $r \ge 2$ (as is in fact already done in \cite{DKW14}). 

Note that \cite[Thm.~1.1]{DKW14} represents 
\[\bar \ell_1 
= c_1 n^{2-\alpha} + n^{1-\alpha+1/\alpha} 
\Big( \frac{\alpha(2-\alpha)(\alpha-1)^{1/\alpha+1}\Gamma(\alpha)}{\Gamma(2-\alpha)^{1/\alpha}} \Big) \varsigma 
+ o_{\mathbb P} \Big( n^{1/\al+1-\al} \Big)
\]  
where $\varsigma$ is a centered, totally asymmetric $\alpha$-stable random variable normalized by the tail decay $\P(\varsigma < - x) \sim x^{-\alpha}$, 
$P(\varsigma > x) = o(x^{-\alpha})$ for $x\to\infty$. 
Thus, comparing with \eqref{S1tail}, 
\begin{equation}
\label{S1varsigma-Umrechnung}
    \mathcal{S}_1 \mathop{=}^d \gamma^{1/\alpha} \mathcal{S}_{1/\gamma} 
\mathop{=}^d -\big(\Gamma(2-\alpha)\big)^{-1/\alpha} \varsigma,
\end{equation}
which explains the coefficient $R_{1,1}$ from \eqref{R11}. See also Remark~\ref{rem:DKW14result} below.

\begin{proof}[Proof of Lemma~\ref{lm:indicators_new}]
We first replace the indicators $\1_{\{\Delta_{l_p}=r_p\}}$ by the conditional expectations $\E[\1_{\{\Delta_{l_p}=r_p\}} | X_{l_p-1}]$. 

Note that on the event $\{\Delta_l=r_p\}$ we have $X_{l}=X_{l-1}-r_p$.
Put
\[
  \widetilde{\tau} := \min \{ i \le n : X_{i} \le r_p+1 \} 
\]
(note $\tau_n - r_p \le \widetilde{\tau} \le \tau_n$)
and let for $1\leq j \leq \widetilde{\tau}$ 
\begin{align*}
S^{(p)}_{j}:= \sum_{l=1}^{j} \frac 1 {X_{l-1}-r_p} \Big(\Pi_{0}^{l-1} \Big)^{r_p} \big(\mathds 1_{\{\D_{l}=r_p\}} - \E[\mathds 1_{\{\Delta_{l}=r_p\}} | X_{l-1}]\big).
\end{align*}
Note that $S^{(p)}$ is a martingale with quadratic variation for $j < \widetilde{\tau}$
\begin{align*}
  \langle S^{(p)}\rangle_j
  & =  \sum_{l=1}^j (X_{l-1}-r_p)^{-2} \left(\Pi_0^{l-1}\right)^{2 r_p} \E\big[ ( \1_{\{\Delta_l=r_p\}} - \mathbb E[ \1_{\{\Delta_l=r_p\}} \mid X_{l-1}])^2
  \,\big|\, X_{l-1} \big] \\
  & \le \sum_{l=1}^j \frac{1}{(X_{l-1}-r_p)^2} \le \sum_{i=X_{j-1}-r_p}^\infty \frac{1}{i^2}
    \le \frac{1}{X_{j-1}-r_p-1} .
\end{align*}
Pick $\varepsilon > 0$ so small that
\begin{equation}
  \label{eq:pickeps}
  2\varepsilon + \frac{1}{2} < \frac{1}{\alpha}.
\end{equation}
Now consider the stopping times $\tau_n^{(b)}:= \min \{ l \le n : X_{l} \le n^{b\varepsilon} \}$ with $1 \le b \le \lfloor 1/\varepsilon \rfloor$. 
They are bounded by $n$, therefore, by means of Doob's maximal inequality
\begin{align*}
  & \mathbb{P}\Big( \max_{ j < \tau_n^{(b)} } |S_j^{(p)}| 
    \ge n^{\varepsilon} \cdot n^{-(b\varepsilon)/2} \Big)
    \le \frac{C}{n^{2\varepsilon} \cdot n^{-b\varepsilon}} \E\Big[ \langle S^{(p)}\rangle_{\tau_n^{(b)}} \Big]
    \le \frac{C}{n^{2\varepsilon}}.
\end{align*}
Also, because of the strong Markov property $\tau_n- \tau_{n}^{(b)} \sim \gamma^{-1} X_{\tau_n^{(b)}}\sim \gamma^{-1}n^{b \varepsilon}$ in probability,
and therefore uniformly for $\tau_n^{(b+1)} \le j < \tau_n^{(b)}$ (with $b \ge 1$ fixed)
\[ S_j^{(p)} =O_{\mathbb P} \Big(n^{\varepsilon -b\varepsilon/2}\Big) =
  O_{\mathbb P}\Big( n^{\varepsilon} \cdot (\tau_n-\tau_n^{(b)})^{-1/2}\Big)= O_{\mathbb P}\Big( n^{2\varepsilon} \cdot (\tau_n-j)^{-1/2}\Big)
  .\]
Choosing $b=1, \ldots, \lfloor 1/\varepsilon \rfloor$ yields
\[S_j^{(p)}= O_{\mathbb P}\Big( n^{2\varepsilon} \cdot (\tau_n-j)^{-1/2}\Big) \]
uniformly 
for all $0 \le j < \tau_n^{(1)}$.
Furthermore, since $\langle S^{(p)} \rangle_{\widetilde{\tau}} \le 1$ we have also have
by analogous arguments
\[ S_j^{(p)} = O_{\mathbb P}\Big( n^{2\varepsilon} \cdot (\tau_n-j)^{-1/2}\Big)
\]
uniformly in $\tau_{n}^{(1)} \le j \le \widetilde{\tau}$.
Combining the above shows that 
\begin{align}\label{eq:S}
  S^{(p)}_{j} =O_{\mathbb P}\Big(n^{2\varepsilon} \cdot (\tau_n-j)^{-1/2}\Big)
  \quad \text{uniformly in } 1 \le j \le \widetilde{\tau}.
\end{align}
(The bound in \eqref{eq:S} is not sharp but will suffice for our
purposes in \eqref{eq:some_error_term} below.)

Note further that by Lemma \ref{lm:DKW} for $1 \le i \le j \le \widetilde{\tau}$
\begin{align*}
\bar S^{(p)}_{i,j}:&=\sum_{l=i}^{j}\frac 1 {X_{l-1}-r_p} \Pi_0^{l-1} 
 =O_{\mathbb P} \Big(\sum_{l=i}^{j}(\tau_n-l)^{-1} \Big(\frac{\tau_n-l}{\tau_n}\Big)^{\alpha-1} \Big) 
 =O_{\mathbb P} \Big(\Big(\frac{\tau_n-i} {\tau_n}\Big)^{\alpha-1} \Big)
\end{align*}
and thus 
\begin{align}\label{eq:Sbar}
\max_{1\leq i\leq j < \tau_n}\bar S^{(p)}_{i,j} =O_{\mathbb P}(1).
\end{align}
It follows that (note that in the second equality we bound the sum of products from above by $\prod_{p=1}^{m-1} \sum_{l=1}^{m-1}  \Big( \Pi_{0}^{l-1} \Big)^{r_p} \cdot \frac 1 {X_{l-1 }-r_p}\cdot \mathds 1_{\{\D_{l}=r_p\}} \,$)
\begin{align}\label{eq:some_error_term}
\sum_{k=1}^{\tau_n-1} &X_k^{1-\al}\Pi_0^{k}  \sum_{1\leq l_1<\dots<l_{m-1}\leq k}\prod_{p=1}^{m-1} \Big( \Pi_{0}^{l_p-1} \Big)^{r_p} \cdot \frac 1 {X_{l_{p}-1 }-r_p}\cdot \mathds 1_{\{\D_{l_p}=r_p\}} \notag \\
& \hspace{2cm}\cdot  \sum_{l_m=l_{m-1}+1} ^k \frac {\Big( \Pi_{0}^{l_m-1}\Big)^{r_m} } {X_{l_m-1}-r_{m-1}} \big(\mathds 1_{\{\D_{l_m}=r_m\}} - \E[\mathds 1_{\{\Delta_{l_m}=r_m\}} | X_{l_{m-1}}]\big) \notag
\\
& =  \sum_{k=1}^{\tau_n-1} X_k^{1-\al}\Pi_0^{k}  \sum_{l_m=m}^k \frac {\Big( \Pi_{0}^{l_m-1}\Big)^{r_m} } {X_{l_m-1}-r_{m-1}} \big(\mathds 1_{\{\D_{l_m}=r_m\}} - \E[\mathds 1_{\{\Delta_{l_m}=r_m\}} | X_{l_{m-1}}]\big) \notag \\
& \hspace{2cm}\cdot
\sum_{1\leq l_1<\dots<l_{m-1}<l_m}\prod_{p=1}^{m-1} \Big( \Pi_{0}^{l_p-1} \Big)^{r_p} \cdot \frac 1 {X_{l_{p}-1 }-r_p}\cdot \mathds 1_{\{\Delta_{l_p}=r_p\}} 
  \notag
\\
& =O_{\mathbb P} \Big( \sum_{k=0}^{\tau_n-1} (\tau_n- k)^{1-\al} \Big( \frac {\tau_n- k}{\tau_n} \Big)^{\alpha-1}n^{2\varepsilon} (\tau_n- k)^{-1/2} \Big) \notag \\
& = O_{\mathbb P} \Big( \tau_n^{1-\alpha+1/2} n^{2\varepsilon} \Big)\notag \\
& = o_{\mathbb P} \Big( n^{1/\al+1-\al} \Big)
\end{align}
(recall the choice of $\varepsilon$ from \eqref{eq:pickeps}).

We obtain that
\begin{align*}
\sum_{k=1}^{\tau_n-1}& X_k^{1-\al} \Pi_0^{k} \sum_{1\leq l_1<\dots<l_m\leq k}\prod_{p=1}^m \Big( \Pi_{0}^{l_p-1} \Big)^{r_p} \cdot \frac 1 {X_{l_{p} -1}-r_p}\cdot \mathds 1_{\{\D_{l_p}=r_p\}}  \\
&= \sum_{k=1}^{\tau_n-1} X_k^{1-\al} \Pi_0^{k} \sum_{1\leq l_1<\dots<l_{m-1}\leq k}\prod_{p=1}^{m-1} \Big( \Pi_{0}^{l_p-1} \Big)^{r_p} \cdot \frac 1 {X_{l_{p} -1}-r_p}\cdot \mathds 1_{\{\D_{l_p}=r_p\}}
\notag \\
& \hspace{4cm}
 \cdot \sum_{l_m=l_{m-1}+1} ^k \frac {\Big( \Pi_{0}^{l_m-1}\Big)^{r_m} } {X_{l_m-1}-r_{m-1}} \E[\mathds 1_{\{\Delta_{l_m}=r_m\}} | X_{l_{m-1}}] \notag \\
& \hspace{0.5cm} + o_{\mathbb P} \Big( n^{1/\al+1-\al} \Big).
\end{align*}
Interchanging the order of summation $m-1$ times it follows by the same argument that
\begin{align*}
\sum_{k=1}^{\tau_n-1}& X_k^{1-\al} \Pi_0^{k} \sum_{1\leq l_1<\dots<l_m\leq k}\prod_{p=1}^m \Big( \Pi_{0}^{l_p-1} \Big)^{r_p} \cdot \frac 1 {X_{l_{p} -1}-r_p}\cdot \mathds 1_{\{\Delta_{l_p}=r_p\}}  \\
&= \sum_{k=1}^{\tau_n-1} X_k^{1-\al} \Pi_0^{k} \sum_{1\leq l_1<\dots<l_m\leq k}\prod_{p=1}^m \Big( \Pi_{0}^{l_p-1} \Big)^{r_p} \cdot \frac 1 {X_{l_{p} -1}-r_p}\cdot  \E[\mathds 1_{\{\Delta_{l_p}=r_p\}} | X_{l_{p-1}}] 
\\
& \qquad + o_{\mathbb P} \Big( n^{1/\al+1-\al} \Big).
\end{align*}
By Lemma \ref{lm:DKW} and \eqref{couplingK12_1} it holds for $1\leq j\leq k$ that
\begin{align*}
\sum_{l=1}^{j} & \frac 1 {X_{l-1}-r_p}\Big(\Pi_{0}^{l-1}\Big)^{r_p} \Big|\, \E[\1_{\{\Delta_{l_p}=r_p\}} | X_{l-1}]-\mathbb P(V_l=r_p) \Big|\\
& \leq \sum_{l=1}^{j} \frac 1 {X_{l-1}-r_p} \Pi_{0}^{l-1} \,\E[\1_{\{\Delta_{l}\neq V_l\}} | X_{l-1}] 
\\
 &= O_{\mathbb P} \Big(\sum_{l=1}^{j} (\tau_n-l)^{-1}\Big(\frac{\tau_n-l}{\tau_n}\Big)^{\alpha-1} \cdot (\tau_n-l)^{-1}\Big)
\\
& = O_{\mathbb P} \Big(\tau_n^{1-\alpha}(\tau_n-j)^{\alpha-2}\Big)
\\
& = O_{\mathbb P} \Big((\tau_n-j)^{-1}\Big).
\end{align*}
Using a similar calculation as in \eqref{eq:some_error_term} and interchanging the order of summation as above we obtain that
\begin{align}\label{eq:approx2}
\bar \ell_r & = \al \Gamma(\al) \sum_{(r_1, \dots, r_m)} \prod_{p=1}^m\widehat r_p \mathbb P(V=r_p)\cdot \sum_{k=1}^{\tau_n-1} X_k^{1-\al} \Pi_{0}^{k} \notag\\
& \hspace{3cm} \cdot\sum_{1\leq l_1<\dots<l_m\leq k}\prod_{p=1}^m \Big( \Pi_{0}^{l_p-1} \Big)^{r_p} \frac 1 {X_{l_{p} -1}-r_p} + o_{\mathbb P} \Big( n^{1/\al+1-\al} \Big).
\end{align}
Moreover, since for $1\leq l \leq k$
\[ \Big|\frac 1 {X_{l-1}-r_p} - \frac 1 {X_{l-1}}\Big| \leq \frac { r_p} {X_{l-1}(X_{l-1}-r_p)}=O_\mathbb P\Big( \frac {1} {(\tau_n- l)(\tau_n- l -r_p)} \Big) \]
it follows that
\begin{align}\label{eq:err1}
\sum_{l=1}^{k} \Big(\frac 1 {X_{l-1}-r_p} - \frac 1 {X_{l-1}}\Big)\Big(\Pi_{0}^{l-1}\Big)^{r_p}&=O_\mathbb P\Big(\sum_{l=1}^{k} \frac {1} {(\tau_n- l)(\tau_n- l -r_p)} \Big(\frac {\tau_n- l}{\tau_n}\Big)^{\al-1}\Big)\notag  \\
&=O_\mathbb P\Big(\tau_n^{1-\al}\sum_{l=1}^{k} (n-\g l)^{\al-3}\Big) \notag
\\
&=O_\mathbb P\Big(\tau_n^{1-\al} (\tau_n- k)^{\al-2}\Big)\notag\\
& = O_{\mathbb P} \Big((\tau_n-k)^{-1}\Big).
\end{align}
Also for $1\leq l \leq k$
\[ \Big|\frac 1 {X_{l-1}} - \frac 1 {X_{l}}\Big| \leq \frac { \Delta_{l}} {X_{l}(X_{l}+\Delta_l)}\leq \frac { \Delta_{l}} {X_{l}^2}\]
and since from Lemma \ref{lm:DKW} $X_l$ is of order $\tau_n-l$ uniformly in $0\leq l<\tau_n$, it follows that
\begin{align*}
\sum_{l=1}^{k} \Big(\frac 1 {X_{l-1}} - \frac 1 {X_{l}}\Big)\Big(\Pi_{0}^{l-1}\Big)^{r_p}&=O_\mathbb P\Big(\sum_{l=1}^{k} \frac {X_{l-1}-X_l} {X_l^2} \Big(\frac {\tau_n-l}{\tau_n}\Big)^{\al-1}\Big)=O_\mathbb P\Big(\sum_{l=1}^{k} \frac {X_{l-1}-X_l} {X_l^2} \Big(\frac {X_l}{\tau_n}\Big)^{\al-1}\Big). 
\end{align*}
Using
\[\sum_{l=1}^{k}  X_{l}^{\al-3}(X_{l-1}-X_l) \leq \int_{X_{k}}^{X_0}x^{\al-3}dx\leq \frac 1 {2-\al} X_k^{\al-2}\] 
we further obtain that
\begin{align}\label{eq:err2}
\sum_{l=1}^{k} \Big(\frac 1 {X_{l-1}} - \frac 1 {X_{l}}\Big)\Big(\Pi_{0}^{l-1}\Big)^{r_p}&=O_\mathbb P\Big( \Big(\frac {X_k^{\alpha-2}}{\tau_n^{\alpha-1}}\Big)\Big) = O_\mathbb P\Big(  (\tau_n-k)^{-1}\Big). 
\end{align}
Moreover note that by Lemma \ref{lm:DKW} that
\begin{align*}
\sum_{l=1}^{k}\frac 1 {X_{l}} \Big(\Big(\Pi_{0}^{l}\Big)^{r_p}-\Big(\Pi_{0}^{l-1}\Big)^{r_p} \Big)& = O_\mathbb P\Big( \sum_{l=1}^{k}  \frac 1 {X_{l}^2} \,\Big(\Pi_{0}^{l}\Big)^{r_p}\Big)
\\
&=O_\mathbb P\Big( \sum_{l=1}^{k} (\tau_n-l)^{-2} \Big(\frac{\tau_n-l}{\tau_n}\Big)^{\alpha-1}\Big) \\
&= O_\mathbb P\Big(  (\tau_n-k)^{-1}\Big). 
\end{align*}
which combined with \eqref{eq:err1} and \eqref{eq:err2}, as above, gives rise to an error of order $o_{\mathbb P}\Big(n^{1-\alpha+1/\alpha}\Big)$. Thus, plugging this in \eqref{eq:approx2} proves the lemma.
\end{proof}

\subsection{Decomposing $\bar\ell_r$ into the deterministic part and a sum of weighted centred jumps}\label{Decomposing}

As before (see Lemma \ref{lm:DKW_partB} part ii)) let
\begin{align}\label{def:Kn}
 K_n:=\Big\lfloor \frac n \gamma -  n^\delta \Big \rfloor\qquad \text{with}\qquad 
\frac 1 \alpha <\delta<1
\end{align}
and consider the length gathered from the leaves up to level $K_n$ and between $K_n$ and $\tau_n$ separately. Note that by \eqref{lm:tau} the probability that $\tau_n> K_n$ converges to 1. We write for $r\in \mathbb N$
\begin{align}\label{eq:ell_bar with L}
\bar \ell_r =\bar \ell_r^{(1)}+\bar \ell_r^{(2)}+ o_{\mathbb P} \Big( n^{1/\al+1-\al} \Big)
\end{align}
with 
\begin{align}\label{eq:ell_bar_i}
\bar \ell_r^{(i)}:= \al \Gamma(\al)  \sum_{(r_1, \dots, r_m)} \prod_{p=1}^m \widehat r_p\mathbb P( V=r_p) \cdot L^{(i)}(r_1, \dots, r_m), \quad i=1,2
\end{align}
and
\begin{align}\label{eq:ell_1}
L^{(1)}(r_1, \dots, r_m):= &  \sum_{k=1}^{K_n} X_k^{1-\al}\Pi_0^{k}  \sum_{1\leq l_1<\dots<l_m\leq k}  \prod_{p=1}^m \frac 1 {X_{l_p}} \Big( \Pi_{0}^{l_p}\Big) ^{r_p}, 
\end{align}
\begin{align}\label{eq:ell_2}
L^{(2)}(r_1, \dots, r_m):= &  \sum_{k=K_n+1}^{\tau_n-1} X_k^{1-\al} \Pi_0^{k} \sum_{1\leq l_1<\dots<l_m\leq k}  \prod_{p=1}^m \frac 1 {X_{l_p}} \Big( \Pi_{0}^{l_p}\Big) ^{r_p}. 
\end{align}

Note that since $X_k=1$ for $k\geq \tau_n$, $\Pi_0^k=0$ for $k\geq \tau_n$ and thus on the event $\{K_n > \tau_n\}$ $L^{(2)}(r_1, \dots, r_m)=0$ and 
the sum in \eqref{eq:ell_1} runs up to $\tau_n-1$.

Before evaluating $L^{(1)}(r_1, \dots, r_m)$ and $L^{(2)}(r_1, \dots, r_m)$ we state a technical lemma.
\begin{lemma}\label{lm:technical}
Let for $s \in \mathbb N$, $x\in [0,1]$ and $b_p>0$ for $p\in \{1, \dots, s\}$ 
\begin{align}\label{def:i}
i_{b_1,\dots,b_s}(x):=\mathop{\int \cdots \int}_{x < y_{s} < \cdots < y_{1}<1} \prod_{p=1}^{s} y_p^{b_p-1} dy_{1}\dots dy_{s}. 
\end{align}
It holds
\begin{align*}
i_{b_1,\ldots,b_s}(x) = a_0 + \sum_{i=1}^{s} a_i x^{b_{s-i+1}+ \cdots + b_s}
\end{align*}
where for $0\leq i\leq s$ (when $i=0$ or $i=s$ the corresponding product is void) 
\begin{align}\label{alphas}
a_i = a_i(b_1, \dots,b_s)=(-1)^i \prod_{j=1}^{s-i} \frac 1{b_j+ \cdots + b_{s-i}} \prod_{k=s-i+1}^{s} \frac 1{b_{s-i+1}+ \cdots + b_k}. 
\end{align}
\end{lemma}

\begin{proof}
Integrating first with respect to $y_s$ in the first multiple integral on the right hand side we obtain
\begin{align}\label{eq:ia}
 \mathop{\int \cdots \int}_{0 < y_{s} < \cdots < y_{1}<1} \prod_{p=1}^{s} y_p^{b_p-1} dy_{1}\dots dy_{s} &=\frac {1} {b_s} \mathop{\int \cdots \int}_{0 < y_{s-1} < \cdots < y_{1}<1} y_{s-1}^{b_s} \prod_{p=1}^{s-1} y_p^{b_p-1} dy_{1}\dots dy_{s-1}.
\end{align}
Iterating we obtain 
\begin{align} \label{valueatzero} 
a_0=i_{b_1,\ldots,b_s}(0)= \frac 1{b_s(b_s+b_{s-1})\cdots(b_s+\cdots + b_1)}.\end{align}
Next, putting $i_{b_1, \ldots, b_{s-1}}(x)=1$ for $s=1$,
\begin{align}\label{iteration} i_{b_1,\ldots,b_s}'(x)=- x^{b_s-1} i_{b_1, \ldots, b_{s-1}}(x) .\end{align}
From this equation via induction on $s$
\[ i_{b_1,\ldots,b_s}(x) = a_s x^{b_1+ \cdots + b_s}+ \cdots+ a_{i}x^{b_{s-i+1}+ \cdots+b_s} + \cdots + a_1 x^{b_s} + a_0, \]
where the coefficients $a_i$ depend on $b_1, \ldots b_s$. We also have
\[ i_{b_1,\ldots,b_{s-1}}(x) = \beta_{s-1} x^{b_1+ \cdots + b_{s-1}}+ \cdots + \beta_1 x^{b_s} + \beta_0.\]
Then because of \eqref{iteration} for $1 \le i \le s$
\[ a_i=-\frac{\beta_{i-1}}{b_{s-i+1}+ \cdots+b_s}.\]
Using the same argument for $\beta_{i-1}$ and iterating we obtain
\[ a_i=i_{b_1, \ldots,b_{s-i} }(0) \prod_{k=s-i+1}^{s} \frac {-1}{b_{s-i+1}+ \cdots + b_k}. \]
Using also \eqref{valueatzero} we get the claim.
\end{proof}

\begin{lemma}\label{lm:L2} 
It holds for $\delta$ as in \eqref{def:Kn}
\begin{align*}
L^{(2)}(r_1, \dots, r_m)=&n^{2-\alpha} \gamma^{-(m+1)} \int_{0}^{\gamma n^{\delta-1}} dx_{m+1} \cdot \mathop{\int \cdots \int}_{x_{m+1} < x_m < \cdots < x_1 < 1}  \prod_{p=1}^m x_p^{(\alpha-1)r_p-1}dx_m \dots dx_1 \\
&
 + n^{1-\alpha+1/\alpha}\frac 1 \gamma \prod_{p=1}^m \frac{1}{r_p + r_{p+1} + \cdots + r_m} \cdot\mathcal\, \mathcal S^{(n)}_{1/\gamma} 
\\
& + o_{\mathbb P}(n^{1-\alpha+1/\alpha}).
\end{align*}
\end{lemma}

\begin{proof}

Using Lemma \ref{lm:DKW} we get for $\varepsilon>0$ using that $l_p \leq k$
\begin{align*}
L^{(2)}&(r_1,  \dots, r_m) \\
&=(\gamma\tau_n)^{1-\alpha}  \sum_{k=K_n}^{\tau_n-1} \Big(1+O_{\mathbb P}\Big((\tau_n-k)^{1/\alpha-1+\varepsilon}\Big)\Big)\\
&\hspace{2cm} \cdot \sum_{1\leq l_1<\dots<l_m\leq k}  \prod_{p=1}^m \frac 1 {\gamma \tau_n} \Big(\frac {\tau_n-l_p} {\tau_n} \Big)^{(\alpha-1)r_p-1} \Big(1+O_{\mathbb P}\Big((\tau_n-l_p)^{1/\alpha-1+\varepsilon}\Big)\Big)
\\
&=(\gamma\tau_n)^{1-\alpha} \sum_{k=K_n}^{\tau_n-1} \Big(1+O_{\mathbb P}\Big((\tau_n-k)^{1/\alpha-1+\varepsilon}\Big)\Big) \sum_{1\leq l_1<\dots<l_m\leq k}  \prod_{p=1}^m \frac 1 {\gamma \tau_n} \Big(\frac {\tau_n-l_p} {\tau_n} \Big)^{(\alpha-1)r_p-1}.
\end{align*}
Replacing the sums by integrals and noting that the multiple sum is uniformly bounded
\begin{align}\label{eq:L2_1}
L^{(2)}&(r_1, \dots, r_m) \notag
\\
& = (\gamma \tau_n)^{1-\alpha} \bigg(\tau_n \Big(\int_{0}^{1-K_n/\tau_n} dx_{m+1}\mathop{\int \cdots \int}_{x_{m+1} < x_m < \cdots < x_1 < 1} \prod_{p=1}^m \frac 1 \gamma x_p^{(\alpha-1)r_p-1}dx_m \dots dx_1 + O_\mathbb P(\tau_n^{-1}) \Big) \notag
\\
& \hspace{2.5cm}+ O_\mathbb P\Big((\tau_n-K_n)^{1/\alpha+\varepsilon}\Big) \bigg) \notag
\\
& = (\gamma \tau_n)^{1-\alpha} \gamma^{-m} \Big(\tau_n\int_{0}^{\gamma n^{\delta-1}} dx_{m+1} \mathop{\int \cdots \int}_{x_{m+1} < x_m < \cdots < x_1 < 1}  \prod_{p=1}^m x_p^{(\alpha-1)r_p-1}dx_m \dots dx_1 \notag \\
&\hspace{2.5cm} + \tau_n \int_{\gamma n^{\delta-1}}^{1-K_n/\tau_n} dx_{m+1} \mathop{\int \cdots \int}_{x_{m+1} < x_m < \cdots < x_1 < 1}  \prod_{p=1}^m x_p^{(\alpha-1)r_p-1}dx_m \dots dx_1 + O_\mathbb P(1) \Big) \notag
\\
&\hspace{1cm} + O_\mathbb P\Big(\tau_n^{1-\alpha}  n^{\delta(1/\alpha+\varepsilon)}\Big).
\end{align}
Consider now the second integral on the right hand side. It holds by Lemma \ref{lm:technical} with $b_p=(\alpha-1)r_p$ for $1\leq p\leq m$ that
\begin{align*}
 \tau_n &\int_{\gamma n^{\delta-1}}^{1-K_n/\tau_n}  dx_{m+1}\mathop{\int \cdots \int}_{x_{m+1} < x_m < \cdots < x_1 < 1}  \prod_{p=1}^m x_p^{(\alpha-1)r_p-1}dx_m \dots dx_1
\\
&=  \tau_n\Big( 1- \frac {K_n} {\tau_n} -\gamma n^{\delta-1}\Big)  \Big( \mathop{\int \cdots \int}_{0 < x_m < \cdots < x_1 < 1}  \prod_{p=1}^m x_p^{(\alpha-1)r_p-1}dx_m \dots dx_1 +o_\mathbb P(1)\Big)
\\
& =\Big(\tau_n - \frac {n} {\gamma} \Big) i_{r_1(\alpha-1), \dots, r_m(\alpha-1)}(0)  + o_\mathbb P(n^{\delta-1+1/\alpha}) .
\end{align*}

Plugging this in \eqref{eq:L2_1}, using Lemma \ref{lm:technical}, Remark \ref{rm:remark} and recalling that $\delta<1$ we get the claim by taking $\varepsilon$ in the estimates from Lemma \ref{lm:DKW} small enough.
\end{proof}

Let $V_1, V_2, \dots$ be the sequence of i.i.d. random variables from the coupling in Section \ref{Preparations}. 
\begin{lemma}\label{lm:L1} 
We have for $\delta$ as in \eqref{def:Kn}
\begin{align*}
 L^{(1)}(r_1, \dots, r_m)&= n^{2-\alpha} \gamma^{-(m+1)} \int_{\gamma n^{\delta-1}}^1 dy_{m+1}
    \hspace{-1.5em} \mathop{\int \cdots \int}_{y_{m+1} < y_m < \cdots < y_1 < 1} \hspace{-1.5em}
    d y_m \cdots d y_1 \: \prod_{p=1}^m y_p^{(\alpha-1) r_p - 1} 
\\
& \quad + n^{1-\alpha+1/\alpha} \sum_{j=1}^{K_n} \frac {V_j -\gamma}{n^{1/\alpha}}  \cdot  F_{\alpha, m,r_1, \dots, r_m}\Big(1-\gamma \frac j n\Big) 
\\
& \quad + o_\mathbb P (n^{1-\alpha +1/\alpha}),\notag
\end{align*} 
where for $x\in [0,1]$ 
\begin{align}\label{eq:F}
& F_{\alpha, m,r_1, \dots, r_m}(x)=C_0\big((r_1,\dots,r_m)\big)+\sum_{k=1}^m C_k\big((r_1,\dots,r_m)\big) x^{(\alpha-1)(r_{k}+\cdots+r_m)}
\end{align}
with $ C_0\big((r_1,\dots,r_m)\big)= (\alpha-1)^{2} \prod_{j=1}^{m} \frac 1{\widehat r_{j-1}-1} $
and $C_k\big((r_1,\dots,r_m)\big)$ as in \eqref{C_k} for $1\leq k\leq m$.
\end{lemma}

\begin{proof}
We use the representations in Lemma \ref{lm:DKW_partB}. Recall the definition of $L^{(1)}(r_1, \dots, r_m)$ in \eqref{eq:ell_1}. It holds (since $\frac{n^{1/\alpha} \mathcal S_{k/n}^{(n)} }{n-\gamma k} \cdot \sum_{j=1}^k \frac{\Delta_j - \gamma }{n-\gamma j}= O_{\mathbb P}\Big(\frac {n^{2/\alpha}} {(n-\gamma k)^2}\Big))$
\begin{align}
  \label{eq:L1_0a}
L^{(1)}&(r_1, \dots, r_m) = \sum_{k=1}^{K_n} X_k^{1-\al} \Pi_0^{k} \sum_{1\leq l_1<\dots<l_m\leq k} \prod_{p=1}^m\Big(\frac 1 {X_{l_p}} \Big( \Pi_0^{l_p}\Big) ^{r_p} \Big)  \\
& = \sum_{k=1}^{K_n} (n-\gamma k)^{1-\alpha} \Big( 1+ (\alpha-1) \frac{ n^{1/\alpha} \mathcal S^{(n)}_{k/n}}{n-\gamma k}
  + O_{\mathbb P}\Big(\frac{n^{2/\alpha}}{(n-\gamma k)^{2} }\Big) \Big
) \notag \\
& \hspace{1cm} 	\cdot \Big(\frac n {n-\gamma k}\Big)^{1 - \alpha}\Big(1-(\alpha -1)\frac{n^{1/\alpha} \mathcal S_{k/n}^{(n)} }{n-\gamma k} + (\alpha -1)\sum_{j=1}^k \frac{\Delta_j - \gamma }{n-\gamma j}+ O_{\mathbb P}\Big(\frac {n^{2/\alpha}} {(n-\gamma k)^2}\Big) \Big) \notag \\
& \hspace{1cm} \cdot  \sum_{1\leq l_1<\dots<l_m\leq k} \prod_{p=1}^m \Big(\frac 1 {n-\gamma l_p } \Big(1+ \frac{ n^{1/\alpha} \mathcal S^{(n)}_{l_p/n}}{n-\gamma l_p}
  + O_{\mathbb P}\Big(\frac{n^{2/\alpha}}{(n-\gamma l_p)^{2} }\Big)\Big)\notag \\
	& \hspace{4cm} \cdot \Big(\frac n {n-\gamma l_p}\Big)^{r_p(1 - \alpha)}\Big(1-r_p(\alpha -1)\frac{n^{1/\alpha} \mathcal S_{l_p/n}^{(n)} }{n-\gamma l_p} 	\notag 
	\\& \hspace{6.5cm}
	+ r_p(\alpha -1) \sum_{j=1}^{l_p} \frac{\Delta_j - \gamma }{n-\gamma j}
+ O_{\mathbb P}\Big(\frac {n^{2/\alpha}} {(n-\gamma l_p)^2}\Big)\Big) \Big).\notag
	\end{align}
Further (recall that  $n^{1/\alpha}\mathcal S_{l_p/n}^{(n)} = \sum_{j=1}^{l_p} (\Delta_j - \gamma )$)
\begin{align}\label{eq:L1_0}
& L^{(1)}(r_1, \dots, r_m) \notag \\
& = n^{1-\alpha}\sum_{k=1}^{K_n} \Big( 1+ (\alpha-1)\sum_{j=1}^k \frac{\Delta_j - \gamma }{n-\gamma j} 
  + O_{\mathbb P}\Big(\frac{n^{2/\alpha}}{(n-\gamma k)^{2} }\Big) \Big
) \notag  \\
& \hspace{1.2cm} \cdot \sum_{1\leq l_1<\dots<l_m\leq k} \prod_{p=1}^m \frac 1 n\Big(\frac{n-\gamma l_p} n\Big)^{r_p(\alpha-1)-1} \Big(1 + \sum_{j=1}^{l_p} (\Delta_j - \gamma)\Big(\frac{r_p(\alpha-1) }{n-\gamma j} -\frac{r_p(\alpha-1) -1}{n-\gamma l_p}\Big) \notag \\
& \hspace{8.5cm}
+ O_{\mathbb P}\Big(\frac{n^{2/\alpha}}{(n-\gamma l_p)^{2} }\Big)\Big).
	\end{align}
Denoting 
\begin{align}
&G_k:= (\alpha-1)\sum_{j=1}^k \frac{\Delta_j - \gamma }{n-\gamma j}  \label{eq:G_k},
\\
&H_q=H_q(l_q,r_q):= \sum_{j=1}^{l_q} (\Delta_j - \gamma)\Big(\frac{r_q(\alpha-1) }{n-\gamma j} -\frac{r_q(\alpha-1) -1}{n-\gamma l_q}\Big) \label{eq:H_q}.
\end{align}	
and then expanding the product over $p$ in \eqref{eq:L1_0} we obtain (using $l_p\leq k$)
\begin{align*}
 L^{(1)}(r_1, \dots, r_m)
& = n^{1-\alpha}\sum_{k=1}^{K_n} \Big( 1+ G_k
  + O_{\mathbb P}\Big(\frac{n^{2/\alpha}}{(n-\gamma k)^{2} }\Big) \Big
) \notag  \\
& \hspace{1.2cm} \cdot \sum_{1\leq l_1<\dots<l_m\leq k} \Big(1+ \sum_{q=1}^m H_{q}
+ O_{\mathbb P}\Big(\frac{n^{2/\alpha}}{(n-\gamma k)^{2} }\Big)\Big)\cdot \prod_{p=1}^m \frac 1 n\Big(1-\gamma \frac{l_p} n\Big)^{r_p(\alpha-1)-1}
	\end{align*}
since by \eqref{eq:S_n} and Lemma \ref{lm:S_n}
\begin{align}
H_q= O_\mathbb P\Big( \frac {n^{1/\alpha}}{n-\gamma l_q} \Big). \label{eq:order_H_q}
\end{align}
Furthermore since $G_k= O_\mathbb P\Big( \frac {n^{1/\alpha}}{n-\gamma k} \Big)$
\begin{align*}
 L^{(1)}(r_1, \dots, r_m)
& = n^{1-\alpha}\sum_{k=1}^{K_n}\sum_{1\leq l_1<\dots<l_m\leq k} \prod_{p=1}^m \frac 1 n\Big(1-\gamma \frac{l_p} n \Big)^{r_p(\alpha-1)-1} \notag
\\ 
& \qquad + 
n^{1-\alpha}\sum_{k=1}^{K_n} G_k \sum_{1\leq l_1<\dots<l_m\leq k} \prod_{p=1}^m \frac 1 n\Big(1-\gamma \frac{l_p} n \Big)^{r_p(\alpha-1)-1}
\notag \\
& \qquad + n^{1-\alpha}\sum_{k=1}^{K_n}\sum_{1\leq l_1<\dots<l_m\leq k} \Big(\prod_{p=1}^m \frac 1 n\Big(1-\gamma \frac{l_p} n \Big)^{r_p(\alpha-1)-1} \sum_{q=1}^m H_{q} \Big)\notag
\\
& \qquad  + O_{\mathbb P}\Big(n^{1-\alpha}\sum_{k=1}^{K_n} \frac{n^{2/\alpha}}{(n-\gamma k)^{2} } \cdot \sum_{1\leq l_1<\dots<l_m\leq k} \prod_{p=1}^m \frac 1 n\Big(1-\gamma \frac{l_p} n\Big)^{r_p(\alpha-1)-1} \Big).
	\end{align*}
 Recall \eqref{def:Kn}. We obtain that
\begin{align*}
n^{1-\alpha}&\sum_{k=1}^{K_n} \frac{n^{2/\alpha}}{(n-\gamma k)^{2} } \cdot \sum_{1\leq l_1<\dots<l_m\leq k} \prod_{p=1}^m \frac 1 n\Big(1-\gamma \frac{l_p} n\Big)^{r_p(\alpha-1)-1}  \notag \\
&= O_{\mathbb P}\Big( n^{1-\alpha+2/\alpha}\sum_{k=1}^{K_n} \frac{1}{(n-\gamma k)^{2} } \cdot \prod_{p=1}^m  \sum_{l=0}^{k-1}   \frac 1 {n-\gamma l} \Big)\notag \\
&= O_{\mathbb P}\Big( n^{1-\alpha+2/\alpha}(n-\gamma K_n)^{-1}(\log n)^m\Big)
\notag \\
&= o_\mathbb P(n^{1-\alpha+1/\alpha})
	\end{align*}
and thus
\begin{align}\label{eq:L1_1c}
 L^{(1)}(r_1, \dots, r_m) & = n^{1-\alpha}\sum_{k=1}^{K_n}\sum_{1\leq l_1<\dots<l_m\leq k} \prod_{p=1}^m \frac 1 n\Big(1-\gamma \frac{l_p} n \Big)^{r_p(\alpha-1)-1} \notag
\\ 
& \qquad + 
n^{1-\alpha}\sum_{k=1}^{K_n} G_k \sum_{1\leq l_1<\dots<l_m\leq k} \prod_{p=1}^m \frac 1 n\Big(1-\gamma \frac{l_p} n \Big)^{r_p(\alpha-1)-1}
\notag \\
& \qquad + n^{1-\alpha}\sum_{q=1}^m \sum_{k=1}^{K_n}\sum_{1\leq l_1<\dots<l_m\leq k} H_{q} \prod_{p=1}^m \frac 1 n\Big(1-\gamma \frac{l_p} n \Big)^{r_p(\alpha-1)-1}\notag
\\
& \qquad + o_\mathbb P(n^{1-\alpha+1/\alpha}),  
	\end{align}

For the first term on the right-hand side it holds (replacing the multiple sum by an integral)
\begin{align}\label{eq:L1_lead0}
n^{1-\alpha}&\sum_{k=1}^{K_n}\sum_{1\leq l_1<\dots<l_m\leq k} \prod_{p=1}^m \frac 1 n\Big(1-\gamma \frac{l_p} n \Big)^{r_p(\alpha-1)-1} \notag
\\
&= n^{2-\alpha} \sum_{k=1}^{K_n}\frac 1 n\sum_{1\leq l_1<\dots<l_m\leq k}\frac 1 {n^m}  \prod_{p=1}^m \Big(1- \gamma\frac{ l_p} n\Big)^{r_p(\alpha-1)-1}
\notag \\
& = n^{2-\alpha}\Big(\int_0^{1/\gamma- n^{\delta-1}} d x_{m+1}
   \mathop{\int \cdots \int}_{0 < x_1 < \cdots < x_m < x_{m+1}} 
    d x_m \cdots d x_1 \: \prod_{p=1}^m (1- \gamma x_p)^{(\alpha-1) r_p - 1} + O( n^{1-\alpha}) \Big)\notag \\
  & = n^{2-\alpha} \gamma^{-(m+1)} \int_{\gamma n^{\delta-1}}^1 dy_{m+1}
    \hspace{-1.5em} \mathop{\int \cdots \int}_{y_{m+1} < y_m < \cdots < y_1 < 1} \hspace{-1.5em}
    d y_m \cdots d y_1 \: \prod_{p=1}^m y_p^{(\alpha-1) r_p - 1} + o\Big(n^{1-\alpha+1/\alpha}\Big) 	
	\end{align}
where we substituted $y_p=1-\gamma x_p$, $p=1,\dots,m+1$. 

We now deal with the second term on the right hand side of \eqref{eq:L1_1c}. Recall \eqref{eq:G_k}. Interchanging the order of summation and replacing sums by integrals we have that
\begin{align*}
n^{1-\alpha} &\sum_{k=1}^{K_n}G_k \sum_{1 \le l_1 < \cdots < l_m \le k}\prod_{p=1}^m \frac1n \Big(1- \gamma \frac{l_p}n\Big)^{r_p(\alpha-1)-1}\\&= (\alpha-1)n^{1-\alpha} \sum_{j=1}^{K_n}\frac{\Delta_j - \gamma}{n-\gamma j} \sum_{k=j}^{K_n}\sum_{1 \le l_1 < \cdots < l_m \le k}\prod_{p=1}^m \frac1n \Big(1- \gamma \frac{l_p}n\Big)^{r_p(\alpha-1)-1} \notag
\\
&=\gamma^{-m-1}n^{1-\alpha} \sum_{j=1}^{K_n}\frac{\Delta_j - \gamma}{n-\gamma j}\sum_{k=j}^{K_n} \Big( \idotsint\limits_{1-\gamma k/n\le y_{m} < \cdots < y_1 \le 1} \prod_{p=1}^m y_p^{(\alpha-1)r_p-1} \, dy_1 \ldots dy_{m} +O(n^{1-\alpha})\Big) 
 \notag
\\
&= \gamma^{-m-1} n^{1-\alpha}  \sum_{j=1}^{K_n}\frac{\Delta_j - \gamma}{n-\gamma j} \frac n\gamma \int_0^{1-\gamma j/n} dy_{m+1} \Big(\idotsint\limits_{ y_{m+1} < \cdots < y_1 \le 1} \prod_{p=1}^m y_p^{(\alpha-1)r_p-1} \, dy_1 \ldots dy_{m} +O(n^{1-\alpha}+n^{\delta-1})\Big) \notag
 \\
&=\gamma^{-m-2}n^{1-\alpha}  \sum_{j=1}^{K_n}\frac{\Delta_j - \gamma}{1-\gamma j/n}\Big(\idotsint\limits_{0\le y_{m+1} < \cdots < y_1 \le 1} \prod_{p=1}^m y_p^{(\alpha-1)r_p-1} \, dy_1 \ldots dy_{m}dy_{m+1} \notag
\\
& \hspace{4.5cm} -\idotsint\limits_{1-\gamma j/n\le y_{m+1} < \cdots < y_1 \le 1} \prod_{p=1}^m y_p^{(\alpha-1)r_p-1} \, dy_1 \ldots dy_{m+1} +O(n^{1-\alpha}+n^{\delta-1})\Big). \notag
\end{align*}
Furthermore
\begin{align}\label{L1_sec_term}
n^{1-\alpha} &\sum_{k=1}^{K_n}G_k \sum_{1 \le l_1 < \cdots < l_m \le k}\prod_{p=1}^m \frac1n \Big(1- \gamma \frac{l_p}n\Big)^{r_p(\alpha-1)-1} \notag
\\
&= n^{1-\alpha+1/\alpha} \sum_{j=1}^{K_n}\frac {\Delta_j - \gamma}{n^{1/\alpha}} \cdot \frac {i_{r_1(\alpha-1), \dots, r_m(\alpha-1), 1}(0)-i_{r_1(\alpha-1), \dots, r_m(\alpha-1), 1}\Big(1-\gamma \frac j n\Big)}{\gamma^{m+2}(1-\gamma \frac j n)} \notag
\\
& \hspace{3cm} + o_{\mathbb P} (n^{1-\alpha+1/\alpha} )\notag
\\
&= n^{1-\alpha+1/\alpha} \frac 1 {\gamma^{m+2} }\sum_{j=1}^{K_n}\frac {\Delta_j - \gamma}{n^{1/\alpha}}  \cdot \Big(d_0+\sum_{i=1}^m d_{i} \Big(1-\gamma \frac j n\Big)^{(\alpha-1)(r_{m-i+1}+\cdots+r_m)}\Big) \notag
\\
& \hspace{1cm} + o_{\mathbb P} (n^{1-\alpha+1/\alpha}).
\end{align}
where the coefficients $d_i=d_i(r_1, \dots, r_m)$, $0\leq i\leq m$, are obtained from Lemma \ref{lm:technical} by setting
\begin{align}\label{def:di_new}
d_i:=-a_{i+1}\big(r_1(\alpha-1),\dots,r_m(\alpha-1),1\big).
\end{align}
Let for $0\leq h_2\leq h_1 \leq m$ (putting the void product always equal to 1) 
\begin{align}\label{P}
P_{h_1,h_2}= \prod_{j=1}^{h_1-h_2} \frac 1{r_j + \cdots + r_{h_1-h_2}} \prod_{k=h_1-h_2+1}^{h_1} \frac 1{r_{h_1-h_2+1}+ \cdots + r_k}.
\end{align}
Then 
\begin{align}\label{di}
d_i = \frac{(-1)^{i} }{(\alpha-1)^{m+1}}\frac 1{r_{m-i+1}+ \cdots + r_m+\frac 1 {\alpha-1}} P_{m,i}, \qquad \text{ for } 0\leq i\leq m.
\end{align}

We now deal with the third term on the right hand side of \eqref{eq:L1_1c}. Let
\begin{equation}\label{eq:b}
b_{n,q}(j,l_q) := \frac{r_q(\alpha-1) }{n-\gamma j} -\frac{r_q(\alpha-1) -1}{n-\gamma l_q}
\end{equation}
for $1\leq q\leq m$ and $1\leq j\leq l_q$.
It holds by interchanging the order of summation and replacing sums by integrals that
\begin{align*}
&n^{1-\alpha} \sum_{k=1}^{K_n} \sum_{1 \le l_1 < \cdots < l_m \le k}\prod_{p=1}^m \frac1n \Big(1- \gamma \frac{l_p}n\Big)^{r_p(\alpha-1)-1} H_q\\
&=n^{1-\alpha}   \sum_{j=1}^{K_n}(\Delta_j-\gamma) \sum_{l_q=j}^{K_n} b_{n,q}(j,l_q) \sum_{k=l_q}^{K_n-1}\Big(\sum_{1 \le l_1 < \cdots < l_q}\prod_{p=1}^{q-1} \frac1n \Big(1- \gamma \frac{l_p}n\Big)^{r_p(\alpha-1)-1}\\& \hspace{6cm}\cdot \sum_{ l_q < \cdots < l_m \le k}\prod_{p=q}^{m} \frac1n \Big(1- \gamma \frac{l_p}n\Big)^{r_p(\alpha-1)-1}\Big)
\\
&= \gamma^{-m+1} n^{1-\alpha}\sum_{j=1}^{K_n}(\Delta_j-\gamma) \sum_{l_q=j}^{K_n}\Big( b_{n,q}(j,l_q) \idotsint\limits_{1-\gamma l_q/n < y_{q-1} < \cdots < y_1 \le 1} \prod_{p=1}^{q-1} y_p^{r_p(\alpha-1)-1} dy_1\ldots dy_{q-1}
\\ 
& \hspace{1cm}\cdot\frac1n  \Big(1- \gamma \frac{l_q}n\Big)^{r_q(\alpha-1)-1}
\\
& \hspace{1.2cm}\cdot \sum_{k=l_q}^{K_n}\,\,\idotsint\limits_{1-\gamma k/n < y_{m} < \cdots < y_{q+1} \le 1-\gamma l_q/n} \prod_{p=q+1}^{m} y_p^{r_p(\alpha-1)-1} dy_{q+1}\ldots dy_{m} +O(n^{-1})\Big)
\\ 
&= \gamma^{-m}n^{1-\alpha}\sum_{j=1}^{K_n}(\Delta_j-\gamma) \sum_{l_q=j}^{K_n} \Big( b_{n,q}(j,l_q) \idotsint\limits_{1-\gamma l_q/n < y_{q-1} < \cdots < y_1 \le 1} \prod_{p=1}^{q-1} y_p^{r_p(\alpha-1)-1} dy_1\ldots dy_{q-1}
\\ 
& \hspace{0.1cm}\cdot  \Big(1- \gamma \frac{l_q}n\Big)^{r_q(\alpha-1)-1}\idotsint\limits_{0 < y_{m+1} < \cdots < y_{q+1} \le 1-\gamma l_q/n} \prod_{p=q+1}^{m} y_p^{r_p(\alpha-1)-1} dy_{q+1}\ldots dy_{m+1}+O(n^{\delta-1})\Big).
\end{align*}
Observe for the second multiple integral on the right hand side that by first integrating with respect to $y_{m+1}$ and iterating we obtain that
\begin{align*}
\idotsint\limits_{0 < y_{m+1} < \cdots < y_{q+1} \le 1-\gamma l_q/n} & \prod_{p=q+1}^{m} y_p^{r_p(\alpha-1)-1} dy_{q+1}\ldots dy_{m+1}
 \\
& =\frac { \Big(1- \gamma \frac{l_q}n\Big)^{(r_{q+1}+\cdots +r_m)(\alpha-1)+1}} {(1+r_m(\alpha-1))\cdots (1+(r_m+\cdots r_{q+1})(\alpha-1))} 
\end{align*}
Therefore
\begin{align}\label{eq:L1_third_term0}
&n^{1-\alpha} \sum_{k=1}^{K_n} \sum_{1 \le l_1 < \cdots < l_m \le k}\prod_{p=1}^m \frac1n \Big(1- \gamma \frac{l_p}n\Big)^{r_p(\alpha-1)-1} H_q\notag
\\
&= \gamma^{-m} n^{1-\alpha}\sum_{j=1}^{K_n}(\Delta_j-\gamma) \sum_{l_q=j}^{K_n} \Big(b_{n,q}(j,l_q) \idotsint\limits_{1-\gamma l_q/n < y_{q-1} < \cdots < y_1 \le 1} \prod_{p=1}^{q-1} y_p^{r_p(\alpha-1)-1} dy_1\ldots y_{q-1}\notag
\\
&\hspace{1.5cm}\cdot  \Big(1- \gamma \frac{l_q}n\Big)^{(r_q+\cdots +r_m)(\alpha-1)} \notag
\\
&\hspace{1.7cm}\cdot
\frac 1 {(1+r_m(\alpha-1))\cdots (1+(r_m+\cdots r_{q+1})(\alpha-1))} +O(n^{\delta-1})\Big)\notag
\\
&= n^{1-\alpha} \frac {(\alpha-1)^m} {(1+r_m(\alpha-1))\cdots (1+(r_m+\cdots +r_{q+1})(\alpha-1))} \notag
\\
&\hspace{1.5cm} \cdot \sum_{j=1}^{K_n}(\Delta_j-\gamma)\Big( r_q\gamma ^{-1} \frac{1}{1-\gamma \frac j n} \sum_{l_q=j}^{K_n} \frac 1 n \Big(1- \gamma \frac{l_q}n\Big)^{(r_q+\cdots +r_m)(\alpha-1)} \notag
\\
&\hspace{4.2cm}\cdot
\idotsint\limits_{1-\gamma \frac{ l_q} n < y_{q-1} < \cdots < y_1 \le 1} \prod_{p=1}^{q-1} y_p^{r_p(\alpha-1)-1} dy_1\ldots dy_{q-1}\notag
\\
& \hspace{3.8cm} - (r_q\gamma ^{-1}-1) \sum_{l_q=j}^{K_n}\frac 1 n\Big(1- \gamma \frac{l_q}n\Big)^{(r_q+\cdots +r_m)(\alpha-1)-1} \notag
\\
&\hspace{4.2cm}\cdot
\idotsint\limits_{1-\gamma \frac{ l_q} n < y_{q-1} < \cdots < y_1 \le 1} \prod_{p=1}^{q-1} y_p^{r_p(\alpha-1)-1} dy_1\ldots dy_{q-1} \Big)\notag
\\
& \quad + o_{\mathbb P} (n^{1-\alpha+1/\alpha}).
\end{align}
Replacing the sums over $l_q$ by integrals gives 
\begin{align}\label{eq:H1}
& \frac{1}{1-\gamma \frac j n}\sum_{l_q=j}^{K_n} \frac 1 n \Big(1- \gamma \frac{l_q}n\Big)^{(r_q+\cdots +r_m)(\alpha-1)}\idotsint\limits_{1-\gamma l_q/n < y_{q-1} < \cdots < y_1 \le 1} \prod_{p=1}^{q-1} y_p^{r_p(\alpha-1)-1} dy_1\ldots dy_{q-1} \notag
\\
&= \gamma^{-1}\frac{1}{1-\gamma \frac j n} \int_0^{1-\gamma j/n}y_q^{(\alpha-1)(r_q+\cdots + r_m)} \idotsint\limits_{y_q < y_{q-1} < \cdots < y_1 \le 1}  \prod_{p=1}^{q-1} y_p^{r_p(\alpha-1)-1} dy_1\ldots dy_{q}  + O(n^{\delta-1}) \notag
\\
&=  \gamma^{-1} \frac{1}{1-\gamma \frac j n}\,\Big(i_{r_1(\alpha-1), \dots, r_{q-1}(\alpha-1), (r_q+\cdots+r_m)(\alpha-1)+1}(0)  \notag
\\
&\hspace{2.7cm}-i_{r_1(\alpha-1), \dots, r_{q-1}(\alpha-1), (r_q+\cdots+r_m)(\alpha-1)+1}\Big(1-\gamma \frac {j} n\Big) \Big)+ O(n^{\delta-1})  \notag
\\
&=  \gamma^{-1} \sum_{i=1}^q e_{q,i} \Big(1-\gamma \frac j n\Big)^{(\alpha-1)(r_{q-i+1}+\cdots+r_m)} + O(n^{\delta-1}),
\end{align}
where the coefficients $e_{q,i}=e_{q,i}(r_1, \dots, r_m)$, $1\leq i\leq q$, are obtained from Lemma \ref{lm:technical} by setting
\begin{align}\label{def:ei_new}
e_{q,i}:=-a_{i}\big(r_1(\alpha-1),\dots,r_{q-1}(\alpha-1),(r_{q}+\cdots+r_m)(\alpha-1)+1\big).
\end{align}
It holds for $1\leq i\leq q \leq m$ that
\begin{align}\label{ei}
&e_{q,i}= \frac{(-1)^{i+1} }{(\alpha-1)^{q}}\frac {r_{q-i+1}+ \cdots + r_q} {r_{1}+ \cdots + r_m +\frac 1 {\alpha-1}}P_{q,i}
\end{align}
with $P_{q,i}$ from \eqref{P}.

Similarly
\begin{align}\label{eq:H2}
& \sum_{l_q=j}^{K_n} \frac 1 n \Big(1- \gamma \frac{l_q}n\Big)^{(r_q+\cdots +r_m)(\alpha-1)-1}\idotsint\limits_{1-\gamma l_q/n < y_{q-1} < \cdots < y_1 \le 1} \prod_{p=1}^{q-1} y_p^{r_p(\alpha-1)-1} dy_1\ldots dy_{q-1} \notag
\\
&= \gamma^{-1}\, \Big(i_{r_1(\alpha-1), \dots, r_{q-1}(\alpha-1), (r_q+\cdots+r_m)(\alpha-1)}(0)- i_{r_1(\alpha-1), \dots, r_{q-1}(\alpha-1), (r_q+\cdots+r_m)(\alpha-1)}\Big(1-\gamma \frac {j} n\Big)\Big) \notag
\\
&\hspace{1cm}+ O(n^{\delta-1})  \notag
\\
&=\gamma^{-1} \sum_{i=1}^q f_{q,i} \Big(1-\gamma \frac j n\Big)^{(\alpha-1)(r_{q-i+1}+\cdots+r_m)} + O(n^{\delta-1}),
\end{align}
where the coefficients $f_{q,i}=f_{q,i}(r_1, \dots, r_m)$, $1\leq i\leq q$, are obtained from Lemma \ref{lm:technical} by setting
\begin{align}\label{def:fi_new}
f_{q,i}:= -a_{i}\big(r_1(\alpha-1),\dots,r_{q-1}(\alpha-1),(r_{q}+\cdots+r_m)(\alpha-1)\big).
\end{align}
It holds for $1\leq i\leq q \leq m$ that
\begin{align}\label{fi}
& f_{q,i}= e_{q,i} \frac {r_{1}+ \cdots + r_m +\frac 1 {\alpha-1}}{r_{1}+ \cdots + r_m}
\end{align}
with $e_{q,i}$ from \eqref{ei}.
Plugging \eqref{eq:H1} and \eqref{eq:H2} in \eqref{eq:L1_third_term0}
\begin{align*}
&n^{1-\alpha} \sum_{k=1}^{K_n} \sum_{1 \le l_1 < \cdots < l_m \le k} H_q\prod_{p=1}^m \frac1n \Big(1- \gamma \frac{l_p}n\Big)^{r_p(\alpha-1)-1}\notag
\\
&= n^{1-\alpha+1/\alpha} \cdot \frac {(\alpha-1)^{m+1}} {(1+r_m(\alpha-1))\cdots (1+(r_m+\cdots +r_{q+1})(\alpha-1))} \notag
\\
& \hspace{1.5cm} \cdot \sum_{j=1}^{K_n}\frac{\Delta_j-\gamma} {n^{1/\alpha} } \cdot \sum_{i=1}^q \Big(1-\gamma \frac j n\Big)^{(\alpha-1)(r_{q-i+1}+\cdots+r_m)} \Big(r_q(\alpha-1)( e_{q,i} -f_{q,i} ) + f_{q,i}\Big) \notag
\\
& \quad + o_{\mathbb P} (n^{1-\alpha+1/\alpha}).
\end{align*}
Plugging this together with \eqref{eq:L1_lead0} and \eqref{L1_sec_term} in \eqref{eq:L1_1c} we get that
\begin{align}\label{eq:L1_almostdone}
 L^{(1)}(r_1, \dots, r_m)&= n^{2-\alpha} \gamma^{-(m+1)} \int_{\gamma n^{\delta-1}}^1 dy_{m+1}
    \hspace{-1.5em} \mathop{\int \cdots \int}_{y_{m+1} < y_m < \cdots < y_1 < 1} \hspace{-1.5em}
    d y_m \cdots d y_1 \: \prod_{p=1}^m y_p^{(\alpha-1) r_p - 1} \notag
\\
& \quad + n^{1-\alpha+1/\alpha} \sum_{j=1}^{K_n} \frac {\Delta_j -\gamma}{n^{1/\alpha}} \cdot  F_{\alpha, m,r_1, \dots, r_m}\Big(1-\gamma \frac j n\Big) \notag
\\
& \quad + o_\mathbb P (n^{1-\alpha +1/\alpha}),
\end{align}
with 
\begin{align*}
F_{\alpha, m,r_1, \dots, r_m}(x)&=  (\alpha-1)^{m+2} \cdot \Big(d_0+\sum_{i=1}^m d_{i} x^{(\alpha-1)(r_{m-i+1}+\cdots+r_m)}\Big) \\
& \qquad + \sum_{q=1}^m \frac {(\alpha-1)^{m+1}} {(1+r_m(\alpha-1))\cdots (1+(r_m+\cdots +r_{q+1})(\alpha-1))} \notag
\\
& \hspace{1.5cm} \cdot  \sum_{i=1}^q x^{(\alpha-1)(r_{q-i+1}+\cdots+r_m)} \Big(r_q(\alpha-1)( e_{q,i} -f_{q,i} ) + f_{q,i}\Big).
\end{align*}

Interchanging the order of summation in the double sum 
\begin{align*}
& F_{\alpha, m,r_1, \dots, r_m}(x)
\\
& = (\alpha-1)^{m+2} \cdot \Big(d_0+\sum_{i=1}^m d_{i} x^{(\alpha-1)(r_{m-i+1}+\cdots+r_m)}\Big) \\
& \qquad + \sum_{k=1}^{m} x^{(\alpha-1)(r_{k}+\cdots+r_m)} \sum_{q=k}^m \frac {(\alpha-1)^{q+1}} {\prod_{h=q+1}^{m}(r_m+\cdots +r_{h}+\frac {1}{\alpha-1})}
\\
&\hspace{6.2cm} \cdot \Big(r_q(\alpha-1)( e_{q,q-k+1} -f_{q,q-k+1} ) + f_{q,q-k+1}\Big)
\\
&=  (\alpha-1)^{m+2} \Big(d_0+\sum_{k=1}^m x^{(\alpha-1)(r_{k}+\cdots+r_m)}\Big( d_{m-k+1} + \sum_{q=k}^m \frac {(\alpha-1)^{q-m-1}} {\prod_{j=q+1}^{m}(r_m+\cdots +r_{j}+\frac {1}{\alpha-1})} 
\\
&\hspace{6.3cm} \cdot \Big(r_q(\alpha-1)( e_{q,q-k+1} -f_{q,q-k+1} ) + f_{q,q-k+1}\Big)\Big)\Big).
\end{align*}
Now note by \eqref{fi} that
\begin{align*}
r_q(\alpha-1)( e_{q,q-k+1} -f_{q,q-k+1} ) + f_{q,q-k+1} = e_{q,q-k+1} \frac {r_1 +\cdots r_{q-1}+r_{q+1}\cdots +r_m + \frac 1 {\alpha-1}}{r_1 +\cdots +r_m}
\end{align*} 
By \eqref{di} we further obtain that
\begin{align*}
F_{\alpha, m,r_1, \dots, r_m}(x)&=  (\alpha-1)^{2} \prod_{j=1}^{m} \frac 1 {r_j+\cdots+r_m}
\\
& \quad +\sum_{k=1}^m x^{(\alpha-1)(r_{k}+\cdots+r_m)}\cdot \frac {\alpha-1 }{r_{j}+\cdots +r_m+\frac 1 {\alpha-1}} \cdot \prod_{j=1}^{k-1} \frac 1 {r_j+\cdots+r_{k-1}} 
\\
&\hspace{1cm}\cdot \Big(\frac {(-1)^{m-k+1}}{\prod_{j=k}^{m} (r_{k}+\cdots +r_j)}
\\
&\hspace{1.5cm} + \frac 1 {r_1+\cdots+r_m}\sum_{q=k}^m (-1)^{q-k} \cdot \frac{r_1 +\cdots r_{q-1}+r_{q+1}\cdots +r_m + \frac 1 {\alpha-1}} {\prod_{j=q+1}^m (r_j+\cdots+r_m+ \frac 1 {\alpha-1} )} 
\\
&\hspace{6.4cm} \cdot  \frac{r_{k} +\cdots  +r_q} {\prod_{j=k}^q (r_j+\cdots+r_k)} \Big),
\end{align*}
which is the form in \eqref{eq:F}.

The last step of the proof is to replace the jumps $\Delta_j$ in the above expression by the coupled random variables $V_1, V_2, \dots$. By \eqref{couplingK12_1} and \eqref{couplingK12_2} the number of uncoupled steps is stochastically bounded by
  $\sum_{j=1}^n B_{j}$ where $B_{1},B_{2},\dots$ are independent, $B_j \sim \mathrm{Ber}(1 \wedge c/j)$ for $c>0$ constant and the probability of a coupling error exceeding $k\geq 1$ is bounded by $ck^{1-\alpha}$. Note that
  \begin{equation}
    \sum_{j=1}^n B_{j} \sim c \log n \quad \text{for } n \to \infty.
  \end{equation}
Thus choosing $c_1 > 1/(\alpha-1)$ we have that 
  \begin{align}\label{eq:coupling}
    & \mathbb{P}\Big( |\Delta_j-V_j|> (\log n)^{c_1} \text{ for some } 0\leq j\leq K_n \, \big|\,
      \text{at most } 2c \log n \text{ uncoupled steps} \Big) \notag \\
    & \hspace{3em} \le C \big( 2c \log n \big) \big( (\log n)^{c_1}\big)^{1-\alpha} \notag
     \\
		& \hspace{3em}  \mathop{\longrightarrow}_{n\to\infty} 0.
  \end{align}
  \smallskip
For such $c_1$ consider the event
\begin{align}
  \label{def:G_n}
 \mathcal G_n = \Big\{ \max_{j=0,1,\dots,K_n} |\Delta_j-V_j| \le (\log n)^{c_1}, \, \text{at most } 2c \log n \text{ uncoupled steps}\Big\}.
\end{align}
Then 
\begin{align}\label{eq:conv_PG_n}
  \mathbb{P}(\mathcal G_n) \to 1
\end{align}
for $n\to\infty$ 
and on $\mathcal G_n$ it holds that
\begin{align*}
  n^{1-\alpha+1/\alpha}  \sum_{j=1}^{K_n} \frac {|\Delta_j-V_j|}{n^{1/\alpha}} \cdot  F_{\alpha, n,r_1, \dots, r_m}\Big(1-\gamma \frac j n\Big) 
  & =  O\Big(n^{1-\alpha+1/\alpha} \sum_{j=1}^{K_n} \frac {|\Delta_j-V_j|}{n^{1/\alpha}} \Big)
  \\
  & = O_\mathbb P\Big(n^{1-\alpha+1/\alpha} 2c\log n \frac {(\log n)^{c_1}}{n^{1/\alpha}} \Big)
  \\
  & =o_\mathbb P\Big(n^{1-\alpha+1/\alpha}\Big),
\end{align*}
which together with \eqref{eq:L1_almostdone} and \eqref{eq:conv_PG_n} gives the claim.
\end{proof}

\section{Proof of Theorem \ref{theorem} and of Corollary \ref{corollary}}\label{sect:Proof}

\begin{proof}[Proof of Theorem \ref{theorem}]

Combining Lemma \ref{lm:L2} and Lemma \ref{lm:L1} we obtain 
\begin{align}\label{eq:L12_final0}
L^{(1)}&(r_1, \dots, r_m) + L^{(2)}(r_1, \dots, r_m) \notag
\\
 &=  n^{2-\alpha}
\gamma^{-(m+1)} \int_{0}^1 dy_{m+1}
    \hspace{-1.5em} \mathop{\int \cdots \int}_{y_{m+1} < y_m < \cdots < y_1 < 1} \hspace{-1.5em}
    d y_m \cdots d y_1 \: \prod_{p=1}^m y_p^{(\alpha-1) r_p - 1} \notag
		\\
		& \qquad + n^{1-\alpha+1/\alpha}  \frac 1 \gamma \prod_{p=1}^m \frac{1}{r_p + r_{p+1} + \cdots + r_m }\mathcal\, S^{(n)}_{1/\gamma} \notag
		\\
		& \qquad + n^{1-\alpha+1/\alpha} \cdot \sum_{j=1}^{K_n} \frac {V_j -\gamma}{n^{1/\alpha}} \cdot  F_{\alpha, m,r_1, \dots, r_m}\Big(1-\gamma\frac j n\Big)
 + o_\mathbb P(n^{1-\alpha+1/\alpha})
\end{align}
with $F_{\alpha, m,r_1, \dots, r_m}$ defined in Lemma \ref{lm:L1}. By Lemma \ref{lm:technical}
\begin{align*}
 \int_{0}^1 dy_{m+1} &
    \hspace{-1.5em} \mathop{\int \cdots \int}_{y_{m+1} < y_m < \cdots < y_1 < 1} \hspace{-1.5em}
    d y_m \cdots d y_1 \: \prod_{p=1}^m y_p^{(\alpha-1) r_p - 1}  
 = \prod_{p=1}^m \frac{1}{(\alpha-1) (r_p + r_{p+1} + \cdots + r_m) + 1} .
\end{align*}
Combining with \eqref{eq:L12_final0} we obtain (recall $\gamma =\frac 1 {\alpha-1}$)
\begin{align}\label{eq:L12_final}
L^{(1)}&(r_1, \dots, r_m) + L^{(2)}(r_1, \dots, r_m) \notag
\\
 &=  n^{2-\alpha}
(\alpha-1) \prod_{p=1}^m \frac{\alpha-1}{(\alpha-1) (r_p + r_{p+1} + \cdots + r_m) + 1} \notag
		\\
		& \qquad + n^{1-\alpha+1/\alpha} \frac 1 \gamma  \prod_{p=1}^m \frac{1}{r_p + r_{p+1} + \cdots + r_m }\mathcal\, S^{(n)}_{1/\gamma} \notag
		\\
		& \qquad + n^{1-\alpha+1/\alpha} \cdot \sum_{j=1}^{K_n} \frac {V_j -\gamma}{n^{1/\alpha}} \cdot  F_{\alpha, m,r_1, \dots, r_m}\Big(1-\gamma\frac j n \Big) 
 + o_\mathbb P(n^{1-\alpha+1/\alpha}).
\end{align}

Plugging this in \eqref{eq:ell_bar_i} and then \eqref{eq:ell_bar_i} in \eqref{eq:ell_bar with L} we obtain
\begin{align}
\bar \ell_r & =\al \Gamma(\al)  \sum_{(r_1, \dots, r_m)} \bigg(\prod_{p=1}^m \widehat r_p \mathbb P( V=r_p) \bigg) \notag
\\
& \hspace{3cm} \cdot \Big(n^{2-\alpha}
(\alpha-1) \prod_{p=1}^m \frac{\alpha-1}{(\alpha-1) (r_p + r_{p+1} + \cdots + r_m) + 1} \notag
		\\
		& \hspace{3.5cm}  + n^{1-\alpha+1/\alpha}  \frac 1 \gamma  \prod_{p=1}^m \frac{1}{r_p + r_{p+1} + \cdots + r_m } \mathcal\, S^{(n)}_{1/\gamma}  \notag
		 \\
		& \hspace{3.5cm} + n^{1-\alpha+1/\alpha} \cdot \sum_{j=1}^{K_n} \frac {V_j -\gamma}{n^{1/\alpha}} \cdot  F_{\alpha, m,r_1, \dots, r_m}\Big(1-\gamma\frac j n\Big)\Big)\notag
\\
& \hspace{1cm}  + o_\mathbb P(n^{1-\alpha+1/\alpha}).
\label{final formula}
\end{align}

Recall the definition of $F_{\alpha, m,r_1, \dots, r_m}$ in Lemma \ref{lm:L1} and the notation therein. We can rewrite 
\begin{align}
F_{\alpha, m,r_1, \dots, r_m}(x)= C_0\big((r_1,\dots,r_m)\big)+\sum_{h=1}^{r-1} x^{(\alpha-1) \cdot h} \cdot \sum_{k=1}^m \delta_{r_k +\cdots+r_m,h} \cdot C_k\big((r_1,\dots,r_m)\big)
\end{align}
and thus \eqref{final formula} becomes
\begin{align}
\bar \ell_r =  &n^{2-\alpha} \cdot \al (\al-1)\Gamma(\al)  \sum_{(r_1, \dots, r_m)} \prod_{p=1}^m \frac{ \widehat r_p \P(V=r_p) \cdot (\alpha-1)}{(\alpha-1) (r_p + r_{p+1} + \cdots + r_m) + 1} \notag
\\
& + n^{1-\alpha+1/\alpha} \cdot \Big(\al \Gamma(\al)  \sum_{(r_1, \dots, r_m)} \Big(\prod_{p=1}^m \Big(\widehat r_p \mathbb P( V=r_p) \Big) \notag
\\
&\hspace{3.2cm} \cdot \Big((\alpha-1)\prod_{p=1}^m \frac{1}{r_p + r_{p+1} + \cdots + r_m } \mathcal\, S^{(n)}_{1/\gamma} + C_0\big((r_1,\dots,r_m)\big)\cdot \sum_{j=1}^{K_n} \frac {V_j -\gamma}{n^{1/\alpha}}\Big)
 \notag
		\\
		& \hspace{2.3cm} + \sum_{h=1}^{r-1}   \al \Gamma(\al)  \sum_{(r_1, \dots, r_m)} \bigg(\prod_{p=1}^m \Big(\widehat r_p \mathbb P( V=r_p)\Big) \cdot \sum_{k=1}^m \Big(\delta_{r_k +\cdots+r_m,h} \cdot C_k \big((r_1,\dots,r_m)\big)\Big) \notag
\\
& \hspace{6cm} \cdot \sum_{j=1}^{K_n} \frac {V_j -\gamma}{n^{1/\alpha}} \Big(1-\gamma\frac j n\Big)^{(\alpha-1) h}\bigg) \notag
\\
&  + o_\mathbb P(n^{1-\alpha+1/\alpha}).
\label{final formula2}
\end{align}
From \eqref{final formula2} we can read off the coefficients of the matrix $R$ given in  \eqref{R_1} and \eqref{R_j}. 

Recall from \eqref{SandVs} that 
\begin{align*}
\mathcal S_{1/\gamma}^{(n)}  = \frac 1{n^{1/\alpha}}\sum_{k=1}^{[n\gamma^{-1}]}(V_k- \gamma) + o_{\mathbb P}(1). 
\end{align*}
Using this and the following lemma together with the Cram\'er-Wold device and \eqref{final formula} gives the claim of the theorem.
\end{proof}

\begin{lemma}\label{lm:limit}
Let $V_1, V_2, \dots$ be i.i.d. random variables in the domain of attraction of a stable law of index $1<\alpha<2$ and having expectation $\gamma$, let $0 < \delta < 1$ and $f : [0,\infty) \to \R$ be a continuous function.
  We have
  \begin{align}
    \label{eq:Sintegrallimit.0}
    & \sup_{n/\gamma - n^\delta < m \le n/\gamma} \Big| \frac{1}{n^{1/\alpha}} \sum_{k=m}^{\lfloor n/\gamma \rfloor}
    f\Big(1-\gamma \frac{k}{n}\Big) (V_k-\gamma)\Big|
      \mathop{\longrightarrow}_{n\to\infty}^d 0. \\
    \intertext{Furthermore,} 
    \label{eq:Sintegrallimit}
    & \sum_{k=0}^{\lfloor n/\gamma - n^\delta \rfloor}
    f\Big(1-\gamma \frac{k}{n}\Big) \frac {V_k -\gamma}{n^{1/\alpha}}
    \mathop{\longrightarrow}_{n\to\infty}^d \int_0^{1/\gamma} f(1-\gamma t) \, dS_t,
  \end{align}
where $S$ is a stable process of index $\alpha$.
\end{lemma}

\begin{proof}
It follows from \cite[Thm.~3.4]{KM86} that
  \begin{align}
    \label{eq:Sintegrallimit.KM}
    & \sum_{k=0}^{\lfloor n/\gamma \rfloor}
    f\Big(1-\gamma \frac{k}{n}\Big) \frac {V_k -\gamma}{n^{1/\alpha}}
    \mathop{\longrightarrow}_{n\to\infty}^d \int_0^{1/\gamma} f(1-\gamma t) \, dS_t.
  \end{align}
  Combining this with \eqref{eq:Sintegrallimit.0} and
  Slutzky's theorem yields \eqref{eq:Sintegrallimit}.
  \smallskip

  For \eqref{eq:Sintegrallimit.0} let
\[ f_n(t):= f\Big(1- \frac \gamma n \Big\lfloor\frac n\gamma\Big\rfloor + \gamma \frac {\lfloor n^\delta \rfloor} n t\Big), \]
$t \ge 0$. Then 
\begin{align*}\Big(\sum_{k=m}^{\lfloor n/\gamma \rfloor}  f\Big(1-  \gamma  \frac kn \Big)(V_k-\gamma) &\Big)_{m=\lfloor n/\gamma\rfloor - \lfloor n^\delta \rfloor, \ldots , \lfloor n/\gamma\rfloor }\\& \stackrel d= \Big(\sum_{k=0}^{m} f_n\Big(\frac k{\lfloor n^\delta\rfloor}\Big)(V_k-\gamma)\Big)_{m=0, \ldots \lfloor n^\delta \rfloor}. 
\end{align*}
Since $f_n(t)\to f(0)$, we may apply  Theorem 3.3 from [KM86] with $\lambda=\lfloor n^\delta \rfloor$ to the right-hand side. In particular the corresponding supremum is of order $O_{\mathbb P}((n^{\delta})^{1/\alpha})=o_{\mathbb P}(n^{1/\alpha})$. This proves formula (3.94).
\end{proof}

\begin{remark}\label{rm:coefficients}
For the coefficient $c_r$ of the leading term $n^{2-\alpha}$ in \eqref{final formula} note that (we use the convention $r_{m+1}=1$)
\begin{align*}
c_r= \al (\al-1)\Gamma(\al)  \sum_{(r_1, \dots, r_m)} \prod_{p=1}^m \frac{ \P(V=r_p) \cdot (\alpha-1)
    \big(r - \sum_{q=1}^p r_{q}\big)}{(\alpha-1) (r_p + r_{p+1} + \cdots + r_m + r_{m+1}) + (2-\alpha)},
\end{align*}
which can be rewritten as
\begin{align*}
c_r&= \frac{\al(\al-1)\Gamma(\al)}{2-\alpha}  \sum_{(r_1, \dots, r_m)} \frac 1 {1+\frac{2-\al}{\al-1}}\P(V=r_m)\notag
\\
& \hspace{4cm}\cdot \prod_{p=2}^m \Big(\frac{  
    r_p+\cdots+r_m+1}{r_p + r_{p+1} + \cdots + r_m + 1 + \frac{2-\al}{\al-1}} \P(V=r_{p-1})\Big)\notag
		\\
	& \hspace{4cm}\cdot \frac{\frac{2-\al}{\al-1}}{r_1+\cdots+r_m+1+ \frac{2-\al}{\al-1}}.
\end{align*}
We can interpret this as $\frac{\al(\al-1)\Gamma(\al)}{2-\alpha}$ times the sum over all path probabilities for the
skeleton chain of a branching process $\xi$ with offspring distribution \eqref{eq:distrib_V} stopped at an independent
exponential time $\tau$ with parameter $\frac{2-\alpha}{\alpha-1}$ to arrive at state $r$ when started in state 1. This is the branching process which appears in \cite[Lemma~30]{BBS07}. 
Thus 
\begin{align*}
c_r&= \frac{\al(\al-1)\Gamma(\al)}{2-\alpha} \, \mathbb P(\xi_\tau=r),
\end{align*}
which is exactly the representation of $c_r$ obtained by Berestycki et.\ al in the proof of their Theorem 9.
\end{remark}

\begin{proof}[Proof of Corollary \ref{corollary}]
Note from \cite[Equation~(33)]{K12} that when $1<\alpha<\frac {1+\sqrt 5} 2$ 
\begin{align}\label{our33}
\frac{ \widetilde c n^{2-\alpha}-L^{(n)}} {n^{1-\alpha +\frac 1 \alpha}}  = \Gamma(\alpha)\alpha(\alpha-1)^\alpha \sum_{j\leq n/\gamma} \Big(\frac j n\Big)^{1-\alpha}\frac{V_j-\gamma}{n^{1/\alpha}}+o_\mathbb P(1).
\end{align}
This combined with 
\cite[Theorem~3.1]{KM88} (note that since $\alpha(1-\alpha)>-1$ for $1<\alpha<\frac {1+\sqrt 5} 2$, condition (A2) therein is fulfilled) 
gives 
\begin{align}\label{last}
\frac{ \widetilde c n^{2-\alpha}-L^{(n)}} {n^{1-\alpha +\frac 1 \alpha}}  \stackrel{d} \longrightarrow  \Gamma(\alpha)\alpha(\alpha-1)^{1+1/\alpha} \int_{0}^{1/\gamma} t^{1-\alpha} d \mathcal S_t. 
\end{align}
To compare literally with \cite[Theorem~3.1]{KM88}  observe also that the upper summation limit in \eqref{our33} is $n/\gamma$ and not $n$. By \cite[Theorem~2.1]{KM88} the integral on the right hand side of \eqref{last} is well-defined when $1<\alpha<\frac {1+\sqrt 5} 2$.

The proof is completed by noting that the representations of $\ell^{(n)}_r$ and $L^{(n)}$ involve the same sequence $V_1, V_2, \dots$ (i.e. the coupling from Section \ref{Preparations} used in the proof of Lemma \ref{lm:L1} is the one from \cite{K12}).

\end{proof}

\begin{remark}[Recovering Theorem~1.1 from \cite{DKW14}]
  \label{rem:DKW14result}
  We restricted the analysis in Section~\ref{Decomposing} to $r\ge 2$, mostly in order to avoid treating special cases
  involving empty sums separately. The approach of Lemmas~\ref{lm:L2} and \ref{lm:L1} can be carried out
  analogously for the case $r=1$, showing that
  \begin{equation}
    \label{eq:barell1rem}
    \al \Gamma(\al) \sum_{k=1}^{\tau_n-1} 
    X_k^{1-\al} \cdot \Pi_{0}^{k}
    = \alpha(\alpha-1) \Gamma(\alpha) n^{2-\alpha} 
    - \alpha(\alpha-1)(2-\alpha)\Gamma(\alpha) \mathcal{S}^{(n)}_{1/\gamma} n^{1-\alpha+1/\alpha}
    + o_{\P}\big( n^{1-\alpha+1/\alpha} \big)
  \end{equation}
  Comparing this with $\bar{\ell}_1$ from \eqref{barell1} confirms the coefficient $R_{1,1}$ from
  \eqref{R11} and in fact provides an alternative proof of \cite[Thm.~1.1]{DKW14}. 
\end{remark}
\begin{proof}[Proof sketch for \eqref{eq:barell1rem}.]
  Analogous to \eqref{eq:ell_bar with L}, decompose the left-hand
  side of \eqref{eq:barell1rem} as $L^{(1)} + L^{(2)}$ with
  \[
    L^{(1)} = \al \Gamma(\al) \sum_{k=1}^{K_n} X_k^{1-\al} \cdot \Pi_{0}^{k}, \quad
    L^{(2)} = \al \Gamma(\al) \sum_{K=K_n+1}^{\tau_n-1} X_k^{1-\al} \cdot \Pi_{0}^{k} .
  \]
  Using Lemma~\ref{lm:DKW_partB} as in \eqref{eq:L1_0a} in the proof of Lemma~\ref{lm:L1}, we have
  (recall $K_n$ from \eqref{def:Kn})
  \begin{align*}
    L^{(1)}
    & = \al \Gamma(\al) \sum_{k=1}^{K_n} (n-\gamma k)^{1-\alpha}
      \Big( 1+ (\alpha-1) \frac{ n^{1/\alpha} \mathcal S^{(n)}_{k/n}}{n-\gamma k}
      + O_{\mathbb P}\Big(\frac{n^{2/\alpha}}{(n-\gamma k)^{2} }\Big) \Big) 
      \notag \\
    & \hspace{3em}
      \cdot \Big(\frac n {n-\gamma k}\Big)^{1 - \alpha}\Big(1-(\alpha -1)\frac{n^{1/\alpha} \mathcal S_{k/n}^{(n)} }{n-\gamma k} + (\alpha -1)\sum_{j=1}^k \frac{\Delta_j - \gamma }{n-\gamma j}+ O_{\mathbb P}\Big(\frac {n^{2/\alpha}} {(n-\gamma k)^2}\Big) \Big) \notag \notag \\
    & = \al \Gamma(\al) n^{1-\alpha} \sum_{k=1}^{K_n}
      \bigg( 1 + (\alpha-1) \sum_{j=1}^k \frac{\Delta_j - \gamma}{n-\gamma j} \bigg) + o_{\P}\big( n^{1-\alpha+1/\alpha} \big) \notag \\
    & = \al (\al-1) \Gamma(\al) n^{2-\alpha} - \al \Gamma(\al) n^{1-\alpha+\delta} \notag \\
    & \hspace{5em} 
      + \alpha (\alpha-1) \Gamma(\alpha) n^{1-\alpha+1/\alpha} \sum_{j=1}^{K_n} \frac{\Delta_j - \gamma}{n^{1/\alpha}}
      \frac{K_n - j}{n-\gamma j}
      + o_{\P}\big( n^{1-\alpha+1/\alpha} \big).
  \end{align*}
  Since $(K_n-j)/(n-\gamma j) = (\alpha-1) \big( 1 - \gamma n^{\delta}/(n-\gamma j)\big)$, this shows that
  \begin{align}
    \label{eq:L(1)r=1}
    L^{(1)}
    & = \al (\al-1) \Gamma(\al) n^{2-\alpha} - \al \Gamma(\al) n^{1-\alpha+\delta}
      + \al (\al-1)^2 \Gamma(\al) n^{1-\alpha+1/\alpha} \mathcal{S}^{(n)}_{1/\gamma} + o_{\P}\big( n^{1-\alpha+1/\alpha} \big).
  \end{align}

  Similarly, using Lemma~\ref{lm:DKW} as in proof of Lemma~\ref{lm:L2}, we see that
  (using also \eqref{eq:tau_n via S_n} in the third equality)
  \begin{align}
    L^{(2)}
    & = \al \Gamma(\al) (\gamma\tau_n)^{1-\alpha}  \sum_{k=K_n}^{\tau_n-1} \Big(1+O_{\mathbb P}\Big((\tau_n-k)^{1/\alpha-1+\varepsilon}\Big)\Big)
      \notag \\
    & = \al \Gamma(\al) (\gamma\tau_n)^{1-\alpha}  \Big( \tau_n - \frac{n}{\gamma} + n^{\delta} \Big)
      \Big(1+ O_{\mathbb P}\Big((\tau_n-n/\gamma + n^{\delta})^{1/\alpha-1+\varepsilon}\Big)\Big)
      \notag \\
    & = \al \Gamma(\al) \Big( n - n^{1/\alpha} \mathcal{S}^{(n)}_{1/\gamma} + o_{\P}\big(n^{1/\alpha}\big) \Big)^{1-\alpha}
      \Big( - \frac{n^{1/\alpha}}{\gamma} \mathcal{S}^{(n)}_{1/\gamma} + n^{\delta} + o_{\P}\big(n^{1/\alpha}\big) \Big)
      + o_{\P}\big( n^{1-\alpha+1/\alpha} \big) \notag \\
    & = \al \Gamma(\al) n^{1-\alpha+\delta} - \al (\al-1) \Gamma(\al) n^{1-\alpha+1/\alpha} \mathcal{S}^{(n)}_{1/\gamma}
      + o_{\P}\big( n^{1-\alpha+1/\alpha} \big)
      \label{eq:L(2)r=1}
  \end{align}
  Adding \eqref{eq:L(1)r=1} and \eqref{eq:L(2)r=1} gives \eqref{eq:barell1rem}
  (note $\al (\al-1)^2 - \al (\al-1) = -\al (\al-1) (2-\al)$).
\end{proof}

\smallskip

\noindent
\textbf{Acknowledgments.}
The authors thank two anonymous referees for their very careful reading of the manuscript and their suggestions which improved the quality of the paper.

\medskip
 The authors were in part supported by the DFG Priority Programme SPP 1590 ``Probabilistic Structures in Evolution'' through projects 221529486 and 221571119 and by the Institute of Mathematics of Gutenberg University Mainz.

\end{document}